\pdfoutput=1
\documentclass[11pt]{amsart}
\usepackage{palatino}

\usepackage[english]{babel}
\usepackage[utf8]{inputenc}
\usepackage{fancyhdr}
\usepackage{amsmath,dsfont}
\usepackage{graphicx} 
\usepackage{caption} 
\usepackage{subcaption} 
\usepackage{amssymb}
\usepackage{amsthm}
\usepackage{dcpic}
\usepackage[hang,flushmargin]{footmisc}
\usepackage[left=1in,right=1in,top=1in,bottom=1in]{geometry}
\usepackage[ocgcolorlinks,linkcolor=linkblue,citecolor=linkred,urlcolor=linkred]{hyperref}
\usepackage{url}
\usepackage{arydshln}

\usepackage{MnSymbol}
\usepackage{pictexwd}
\usepackage{tikz-cd}
\usepackage{verbatim}
\usepackage{xcolor}
\usepackage{bibentry}
\usepackage[normalem]{ulem}

\usepackage{mdframed}

\usepackage{todo}
\makeatletter
\renewcommand*\@tododisplay[0]
\makeatother
\let\TodoOriginal\Todo
\renewcommand\Todo[2][To~do]{%
  {\TodoOriginal[#1]{\normalsize #2}}%
  \@todotrue%
}

\usepackage{JH-abbrev}

\usepackage{enumitem}
\usepackage{empheq}

\usepackage{pgf}

\makeatletter
\newcommand\pgfinvisible{\pgfsys@begininvisible}
\newcommand\pgfshown{\pgfsys@endinvisible}
\makeatother

\def\CIpidisplay{hide}

\usepackage[final]{showkeys} 
\usepackage{seqsplit} 
\usepackage{etoolbox}

\makeatletter
\newcommand{\affiliations}[1]{\gdef\@affiliations{#1}}
\let\@affiliations\@empty
\patchcmd{\@setauthors}
  {\MakeUppercase{\authors}}
  {\MakeUppercase{\authors}\par\smallskip\normalfont\@affiliations}
  {}
  {\PackageError{Project Regular Draft}{Unable to add affiliations below the authors}{}}
\makeatother

\usepackage{calc} 

\newcommand{\fullrlap}[1]{%
  \rlap{\kern\dimexpr-\@totalleftmargin+\textwidth+\marginparsep\relax#1}}

\newcommand{\subtag}[1]{
  \makeatletter
  \def\@currentlabel{#1}
  \makeatother
  \renewcommand{\theequation}{#1\arabic{equation}}
}

\newcommand{\e}{\mathrm{e}}
\newcommand{\bsigma}{\boldsymbol{\sigma}}

\usepackage{placeins}
\usepackage{thmtools}
\usepackage{cleveref}
\usepackage{environ}
\usepackage{booktabs}
\usepackage{tabularx}
\NewEnviron{killcontents}{}

\numberwithin{equation}{section}

\definecolor{linkred}{rgb}{0.75,0,0}
\definecolor{linkblue}{rgb}{0,0,0.75}

\theoremstyle{plain}

\theoremstyle{plain}
\newtheorem{theorem}{Theorem}[section]
\newtheorem{lemma}[theorem]{Lemma}
\newtheorem{proposition}[theorem]{Proposition}
\newtheorem{corollary}[theorem]{Corollary}

\theoremstyle{definition}
\newtheorem{definition}[theorem]{Definition}
\newtheorem{remarkinner}[theorem]{Remark}
\newenvironment{remark}
  {\pushQED{\qed}\remarkinner}
  {\popQED\endremarkinner}
\newtheorem{reminner}[theorem]{Remark}
\newenvironment{rem}
  {\pushQED{\qed}\reminner}
  {\popQED\endreminner}
\newtheorem{example}[theorem]{Example}

\newcolumntype{L}[1]{%
  >{\raggedright\arraybackslash}p{#1}}
\newcolumntype{Y}{%
  >{\raggedright\arraybackslash}X}

\Crefname{proposition}{Proposition}{Propositions}
\Crefname{definition}{Definition}{Definitions}
\Crefname{lemma}{Lemma}{Lemmas}

\newcommand{\aut}{\mathrm{Aut}}

\definecolor{kellygreen}{rgb}{0.3, 0.73, 0.09}

\newcommand{\sss}{\scriptscriptstyle}

\begin{document}

\title{A second-moment proof of quenched equals annealed\\for the Potts model on random regular graphs}

\author[Remco van der Hofstad]{Remco van der Hofstad\textsuperscript{1}}
\author[Kai-Chun Huang]{Kai-Chun Huang\textsuperscript{1}}
\author[Noela Müller]{Noela Müller\textsuperscript{1}}
\author[Wioletta M. Ruszel]{Wioletta M. Ruszel\textsuperscript{2}}

\affiliations{%
\textsuperscript{1}Department of Mathematics and Computer Science, Eindhoven University of Technology, The Netherlands\\
\textsuperscript{2}Mathematics Institute, Utrecht University, The Netherlands}


\date{}

\begin{abstract}
Ferromagnetic spin models on random regular graphs, such as the Ising and Potts models, have been studied extensively.
Several equivalent variational and recursive representations of their limiting pressure are known, and these representations can be used to establish agreement between the quenched and annealed pressures. 
We give a concise overview of these representations and highlight the connections among the different formulations.
We then derive analogous representations for the exponential growth rate of the {\em second moment} of the partition function, and prove that it has the same order as the square of the annealed pressure, up to a subexponential error.
Combined with concentration of the quenched pressure, this identity yields a transparent second-moment proof of quenched--annealed agreement on random regular graphs. 
By bringing these representations and arguments together in a unified presentation, we aim to provide an accessible entry point to the literature on ferromagnetic spin models on random regular graphs.
\end{abstract}
\maketitle


\section{Introduction}
Spin systems provide some of the most fundamental models of collective behavior in statistical mechanics~\cite{Niss2005,lenz1920,Potts}. Despite their simple local interaction rules, they exhibit rich macroscopic phenomena, including phase transitions, symmetry breaking and critical behavior. Understanding how these phenomena emerge from the interplay between interactions and underlying graph geometry remains a central theme in modern probability and mathematical physics. 

In particular, spin systems on \emph{random graphs} have attracted considerable interest  \cite{Phys,Dembo,Dom,Sun,can,Rem,Dembo2,Dembo_rcm,Remco2,RemcoLDP,GerschenfeldMontanari,BasakDembo2017,MontanariMosselSly2012} because they provide a framework for understanding the interplay between structural randomness and collective behavior. Such models are subject to \emph{two} sources of randomness: the graph structure and the spin configuration defined on it. 
For a broad treatment of stochastic processes on random graphs, including Ising and Potts models, see, e.g.,~\cite{vdH_SPoRG}.

The main spin system considered in this article is the ferromagnetic $q$-state Potts model, where $q \geq 2$ is an integer. Its special case $q=2$ corresponds to the Ising model, originally proposed by Lenz as a model of ferromagnetism with binary spins and short-range interactions~\cite{lenz1920}. 
Potts later introduced the $q$-state generalization, in which each spin can take more than two values~\cite{Potts1952,Potts}.

To formally define the model, let $G=(V,E)$ be a finite graph. The corresponding $q$-state Potts measure on $G$ at inverse temperature $\beta\geq 0$  and with external field $B\geq 0$ is defined by\footnote{Here and in the following, we abbreviate $[q] = \{1,\ldots,q\}$.}
\begin{equation} \label{eq_def_Potts}
	\mu^{\mathrm{Potts}}_{\beta, B}({\bsigma})
	:= \frac{1}{Z_G(\beta,B)}
	\exp\left\{
	\beta\sum_{\{i,j\}\in E}\mathds{1}_{\{\sigma_i=\sigma_j\}}
	+B\sum_{i\in V}\mathds{1}_{\{\sigma_i=1\}}
	\right\},
	\qquad \bsigma\in[q]^V, 
\end{equation} 
where the normalization constant $Z_G \equiv Z_G(\beta, B)$ is called the \textit{partition function}. Each $\bsigma\in [q]^V$ is called a \emph{spin configuration}.
The parameter $\beta$ controls the strength of the interaction between neighboring spins.
Since $\beta\geq 0$, configurations in which adjacent vertices have the same spin receive greater weight and are therefore favored. Thus, the model is \emph{ferromagnetic}.
The external field $B$ favors the distinguished spin value $1$.

The main purpose of this paper is to provide a concise and unified overview of existing results on ferromagnetic Potts models on random $d$-regular graphs.  {Our primary random graph model is the random $d$-regular multi-graph generated by the configuration model. Assume that $dn$ is even. Starting from $n$ vertices, each vertex is assigned $d$ half-edges and the resulting $dn$ half-edges are then paired uniformly at random. By construction, every vertex has degree $d$, although self-loops and multiple edges may occur. See~\Cref{section:prelim} for more detailed definition.} {Throughout, we consider the thermodynamic limit $n\to\infty$ with $d$ fixed. }


The main quantity {of interest in this setting} 
is the \emph{pressure (per particle)}. {The double-randomness of graph and spin configuration has the effect that the thermodynamic behavior of these systems can be studied under two notions of averaging: the \emph{quenched} and the \emph{annealed} settings. }
Given a finite random graph $G=(V,E)$, the finite-volume
\emph{quenched} and \emph{annealed pressure} are, respectively, given by
\begin{equation}
	\varphi^{\mathrm{qu}}_G(\beta,B)
	:= \frac{1}{|V|}\bbE[\log Z_G(\beta,B)],
	\qquad
	\varphi^{\mathrm{an}}_G(\beta,B)
	:= \frac{1}{|V|}\log\bbE[Z_G(\beta,B)].
\end{equation}
Informally, the quenched pressure describes the partition function on typical graph realizations, whereas the annealed pressure averages over the graph and spin randomness simultaneously, and may be dominated by some atypical graph event.
We define the limiting pressures and annealed pressures $\varphi^{\mathrm{qu}}$ and $\varphi^{\mathrm{an}}$ as the limits of the finite-volume pressures when the graph size $|V|$ tends to infinity, whenever these limits exist.

{Many of the foundational rigorous results for Potts models on random graphs were first obtained in the special case $q=2$, corresponding to the Ising model.} 
In the quenched setting, the existence of the limiting pressure for the ferromagnetic Ising model on random regular graphs was first established by Gerschenfeld and Montanari~\cite{GerschenfeldMontanari}.
Dembo and Montanari~\cite{Dembo} subsequently obtained the limiting pressure and its cavity representation for the broader class of locally tree-like graphs, and this approach was extended in \cite{Dom} to power-law random graphs, which remain locally tree-like, but whose limiting degree distribution has a polynomially decaying tail and may have infinite variance.
It was later generalized in \cite{Dembo2} to factor models, which include the Potts model as a special case, on locally tree-like sparse random graphs.
The limiting quenched pressure thus admits a unified treatment through local convergence for a broad class of sparse random graphs and factor models, whereas the annealed pressure generally requires model-specific arguments. See, e.g., \cite{can,Remco2,RemcoLDP}.

In this article, we review several representations of the annealed and quenched pressures for the $q$-state Potts model with $q \ge 2$. We explain the main ideas behind these representations and the connections between them, with the aim of providing an accessible entry point to the subject.


A central question in the study of disordered systems that we will address is whether the quenched and annealed pressures coincide. While it is an immediate consequence of Jensen's inequality that the quenched pressure is at most the annealed pressure, it may of course be strictly smaller. For ferromagnetic Potts models on random regular graphs, this question is by now well understood, but the different approaches to it highlight several important probabilistic, combinatorial, and variational ideas.



{Since the quenched pressure is the expectation of $\log Z_G(\beta,B)$, while the annealed pressure involves $\log \mathbb{E}[Z_G(\beta,B)]$, the quenched-annealed agreement problem is naturally linked to the concentration properties of the partition function $Z_G(\beta,B)$.} Generally, for a nonnegative random variable, comparing its first and second moments is a standard way to control fluctuations and determine whether its expectation reflects typical behavior.
Analogously, {with the aim of providing a simple approach to the quenched-annealed agreement problem,} we study $\bbE[Z_{G_n}(\beta,B)^2]$ and $\bbE[Z_{G_n}(\beta,B)]$ as $G_n = (V_n, E_n)$ grows.
However, partition functions typically
grow exponentially in $|V_n|$, and even relatively small fluctuations in
$Z_{G_n}(\beta,B)$ may potentially lead to a non-trivial discrepancy between
$\mathbb{E}[Z_{G_n}(\beta, B)^2]$ and $\mathbb{E}[Z_{G_n}(\beta, B)]^2$ as $|V_n|$ grows. Since the pressure only records the exponential growth
rate of $Z_{G_n}$, the relevant comparison must be made on the \emph{exponential scale}, i.e., we ask whether 
\begin{equation} \label{eq_expo_growth}
	\bbE[Z_{G_n}(\beta,B)
    ^2] \overset{?}{=}e^{o({|V_n|})}
 \bbE[Z_{G_n}(\beta,B)]^2, \quad \text{as }n\to\infty.
\end{equation}

{We therefore investigate the asymptotic exponential growth of $\mathbb{E}[Z_{G_n}^2]$ and its relation to the annealed pressure.
The following theorem gives an informal version of our main result.}

\begin{theorem}
	[Quenched equals annealed: Informal version] \label{thm_main_inf} Let $d \ge 2$, $\beta\geq 0, B\geq 0$.
	For the ferromagnetic Ising and Potts models  on both the random $d$-regular configuration model and the uniform simple $d$-regular graph {$G_n$}, the second moment of the partition function has the same exponential growth rate as the square of its first moment:
	\begin{equation}\label{e:second_moment}
\bbE[Z_{G_n}(\beta,B)^2] = \e^{o({|V_n|})}
 \bbE[Z_{G_n}(\beta,B)]^2, \quad \text{as }n\to\infty.
	\end{equation}
	Moreover, for each graph model, the limiting quenched and annealed pressures exist and coincide. 
	The analogous conclusions also hold for the larger class of ferromagnetic spin models\footnote{{For the general definition of a ferromagnetic spin model, see Definition \ref{d:ferro} below.}}. 
\end{theorem}

The precise statement of Theorem \ref{thm_main_inf} is given in \Cref{t:main1} and \Cref{t:main2} below.
The asymptotic relation \eqref{e:second_moment} is not stated explicitly in \cite[Theorem~7]{BIS}, which instead concerns the partition function restricted to a dominant empirical spin distribution. Nevertheless, the proof passes through the unrestricted paired-spin model and contains the matrix-norm calculation from which the full second-moment identity may be deduced; see \cite[proof of Theorem~7, equations~(22)--(25)]{BIS}. We make this implication explicit in \Cref{p:secondmoment_BIS} and obtain the corresponding statement for the unrestricted partition function.

Quenched--annealed agreement has previously been established under various assumptions and at different levels of generality \cite{BIS,Dembo2,Sun,can,BasakDembo2017,Remco_rcm,Hungarian}. Our contribution is to highlight a particularly direct route to this agreement through the full second moment. Specifically, we derive its exponential growth rate from the paired-spin representation and the matrix-norm argument of \cite{BIS}, combine this with concentration of the partition function, and transfer the resulting conclusions from the configuration model to the uniform simple $d$-regular graph. In doing so, we isolate the second-moment mechanism underlying quenched--annealed agreement and clarify how the restricted-phase analysis of \cite[Theorem~7]{BIS} naturally leads to a statement for the full partition function.


\medskip

\paragraph{\bf Outline}
The remainder of this paper is organized as follows.  In \Cref{section:prelim}, we {introduce basic notation and} define the physical and random graph models we are interested in. We also state Frieze's trick, which serves as a motivation for the study of the second moment {in the context of spin systems}. In \Cref{subsection:regular_moment}, we review the relevant representations of the annealed pressure for general factor models and derive
analogous representations for the normalized logarithm of the second moment.
In \Cref{subsection:regular_compare}  we apply these representations to random $d$-regular graphs,
identify the second-moment exponent, and deduce {the} quenched--annealed agreement.

\section{Preliminaries}\label{section:prelim}
\subsection{Notation}


For $n \in \bbN$, we write $[n] := \{1,\dots ,n\}$.
For $\boldsymbol{z} \in \bbR^n$ and $p\in[1,\infty)$, we define its $p$-norm by
\[
  ||\boldsymbol{z}||_p := \left( \sum_{i \in [n]}  |z_i|^p\right)^{1/p}. 
\] 
Given a matrix $\mathbf{A} \in \bbR^{m\times n}$ and $p,q\in[1,\infty)$, we define its induced $||\cdot||_{p\to q}$-norm by
\begin{equation}
	||\mathbf{A}||_{p\to q} := \sup_{ \boldsymbol{z}\in \bbR^n : \boldsymbol{z} \ne 0 }\frac{||\mathbf{A}\boldsymbol{z}||_q}{||\boldsymbol{z}||_p}.
\end{equation}
We write $\mathbf{A} \succeq 0$ if $\mathbf{A}$ is positive semidefinite. 
\begin{definition}[Entropy] 
The \textit{Shannon entropy} $H$ of a discrete probability distribution $ \boldsymbol{p} = (p_x)_{x\in S}$ on a countable set $S$ is defined {as}
\[
H(\boldsymbol{p}) =- \sum_{x \in S} p_x \log p_x,
\]
with the convention that $0 \log 0 \equiv 0$. 
\end{definition}
For a finite set $\calX$, define the probability simplex $\calP_{\calX}$ by
\[
\calP_{\calX}
:= \left\{\boldsymbol{p}\in[0,1]^{\calX}:\sum_{x\in\calX}p_x=1\right\}.
\]
In particular, we write $\calP_n:=\calP_{[n]}$.
For $r\in\bbN$, we also define
\[
\calP_{r\times r}
:= \left\{\mathbf{x}\in[0,1]^{r\times r}:
x_{ij}=x_{ji}\ \text{and}\ \sum_{i,j=1}^r x_{ij}=1
\right\}.
\]
Throughout the paper, $\bbP$ and $\bbE$ denote probability and expectation with respect to the randomness of the underlying random graph, unless another source of randomness is explicitly indicated. For random variables $X_n$ and $X$, we write $X_n \overset{\bbP}{\longrightarrow} X$ if $X_n$ converges to $X$ in probability, that is, if $\bbP(|X_n-X|>\ve)\to 0$ for every $\ve>0$.
A finite (multi-)graph will be denoted by $G = (V,E)$, and we write $v \sim u$ if $v, u\in V$ are connected by an edge. We also denote the graph distance between $u$ and $v$ by $d_G(u,v)$. For any $r \ge 0$, the ball centered at $v \in V$ of radius $r$, denoted by $B^{\sss(G)}( v; r)$, is defined as the subgraph of $G$ with the following vertex and edge sets:
\begin{equation}
    \begin{split}
        &   V(B^{\sss(G)}( v; r)) := \{ u \in V: d_G(u,v) \le r\} ,\\
        &E(B^{\sss(G)}( v; r)):= \{ \{u,w\}  \in E: d_G(u,v), d_G(w,v) \le r\}.
    \end{split}
\end{equation}

\subsection{Spin models and the quenched-annealed distinction}

{In this section, we introduce the spin models considered in this work and define the quenched and annealed settings, together with their associated pressures. We conclude by introducing the normalized logarithmic second moment of the partition function.}


{Let $G=(V,E)$ be a finite graph. Given a finite spin set (or alphabet) $\mathcal X$, the associated configuration space is \[ \Omega \equiv \Omega_G := \mathcal X^V. \] Its elements $\bsigma\in \Omega$ are called \textit{spin configurations}.} For the $q$-state Potts model, we set $\calX = [q]$. 

{In the literature, general spin models on finite graphs have been formalized in different ways, two of which we present here: the \emph{factor model} and the \emph{matrix model}. These generalize the Ising and Potts models \cite{Sun}, and we will obtain representations of the first and second moments of the partition function $Z_G(\beta,B)$ from \eqref{eq_def_Potts} regarded as a factor model in \Cref{subsection:regular_moment}.}


\begin{definition}[Factor model \cite{Dembo2}] \label{d:factor} {Let $G = (V,E)$ be a finite graph, } $\psi: \calX \times \calX \to (0,\infty)$ be a symmetric function, and $\bar{\psi}: \calX \to (0,\infty)$. {The \textit{factor model} specified by the pair $\pmb{\psi} =(\psi, \bar{\psi})$ is the probability measure $\mu_{\pmb{\psi}}$ on $\calX^V$ given by}
	\begin{equation}
		\mu_{\pmb{\psi}}({\bsigma}) = \frac{1}{Z_{G,\pmb{\psi}}}\prod_{ \{i,j\}\in E}\psi(\sigma_i,\sigma_j) \prod_{i \in V}\bar{\psi}(\sigma_i), \qquad \bsigma \in \calX^V.
	\end{equation}  
	Here, $Z_{G,\pmb{\psi}}$ is the normalization constant that makes $\mu_{\pmb{\psi}}$ a probability measure. 
    $Z_{G,\pmb{\psi}}$ is called the \textit{partition function} of the factor model.
\end{definition}
{In \Cref{d:factor},} we exclude the case where $\psi$ and $\bar{\psi}$ {can} take the value zero. This positivity condition can be relaxed to \textit{permissive} specifications \cite{Dembo2,Sun}, but since our focus is on ferromagnetic Potts models, we present all results for this slightly more restrictive setting. As mentioned, the 
 ferromagnetic $q$-state Potts model with a uniform external field $B$ and inverse temperature $\beta$ {is a factor model with specification}
\begin{align} \label{def_spec_Potts}
    \psi(a,b)=\exp\left(\beta\mathds{1}_{\{a=b\}}\right),
    \qquad
    \bar\psi(a)=\exp\left(B\mathds{1}_{\{a=1\}}\right).
\end{align}
{A second common representation of a slightly more restrictive}  class of spin models is 
in
terms of an interaction matrix. See, for example, \cite{BIS}. We refer to
this formulation as the \emph{matrix model}:

\begin{definition}[{(Ferromagnetic) Matrix model} \cite{BIS}]\label{d:ferro} 
Let $G = (V,E)$ be a finite graph and  $\mathbf{B} = (B_{i,j})_{i,j \in \calX }$ a symmetric matrix  with positive entries. 
The matrix model specified by the \emph{interaction matrix} $\mathbf{B}$ is the probability measure $ \mu_{\mathbf{B}}$ on $\calX^V$ given by
  \[
  \mu_{\mathbf{B}}({\bsigma})
  :=\frac{1}{Z_{G,\mathbf{B}}} \prod_{\{i,j\} \in E} B_{\sigma_i,\sigma_j},
  \qquad \bsigma\in\calX^V.
  \]
 {Here, $Z_{G,\mathbf{B}} $ is the normalization constant that makes $\mu_{\bfB}$ a probability measure.} The model is called \emph{ferromagnetic} if $\mathbf{B}$ is positive semidefinite.
\end{definition}
In \Cref{d:ferro}, for simplicity, we assume that every entry of $\mathbf{B}$ is strictly positive, while Galanis et al.~\cite{BIS} consider non-negative \emph{ergodic}\footnote{We refer to \cite[Chapter 6]{HornJohnson2013} for the definition of ergodic matrices.} matrices. Clearly, every real symmetric matrix whose entries are strictly positive is ergodic. 

\begin{example}[Potts model without external field] 
{The ferromagnetic $q$-state Potts model on any finite graph $G$ with $B=0$ can be captured as a ferromagnetic matrix model:} 
It is easy to check that the interaction matrix 
\begin{equation}\label{def_B}
	\mathbf{B} = \begin{pmatrix}
\e^\beta & 1       & \cdots & 1 \\
1       & \e^\beta & \cdots & 1 \\
\vdots  & \vdots  & \ddots & \vdots \\
1       & 1       & \cdots & \e^\beta
\end{pmatrix} \in \bbR^{q \times q}
\end{equation}
{specifies the $q$-Potts model with zero external field on an arbitrary finite graph. We next argue that $\mathbf{B}$ is positive semidefinite.} Write $\mathbf{J}_q$ for the $q\times q$ matrix whose entries are all $1$. {As $\boldsymbol{x}^T\mathbf{J}_q\boldsymbol{x} \geq 0$ for all $\boldsymbol{x} \in \bbR^q$, $\mathbf{J}_q$ } is positive semidefinite, {and} $\mathbf{B} = \mathbf{J}_q + (\e^\beta - 1)\mathbf{I}_q$, where $\mathbf{I}_q$ is the identity matrix of dimension $q$. Since $\beta\ge 0$, we have $\e^\beta - 1 \ge 0$. Hence, $ (\e^\beta - 1) \mathbf{I}_q$ {as a {non-negative} multiple of a positive definite matrix is positive {semi}definite}. Since $\mathbf{B}$ is a sum of two positive semidefinite matrices, we deduce that $\mathbf{B}$ {is positive semidefinite}. 
\end{example}

\begin{example}[Potts model with nonzero field] \label{example:concentration}
	
Consider a ferromagnetic Potts model on a $d$-regular graph with general {$B > 0$.}
It can be checked that the interaction matrix 
\begin{equation}
	\mathbf{B}_B = \begin{pmatrix}
\e^{\beta + 2B/d}&   \e^{B/d}    & \cdots &  \e^{B/d}\\
  \e^{B/d}     & \e^\beta & \cdots & 1 \\
\vdots  & \vdots  & \ddots & \vdots \\
 \e^{B/d}    & 1       & \cdots & \e^\beta
\end{pmatrix} \in \bbR^{q \times q}
\end{equation}
specifies this model. Moreover, {with $\mathbf{B}$ as in \eqref{def_B} and }
\begin{equation}
	\mathbf{D}= \begin{pmatrix}
\e^{B/d} & 0 & \cdots & 0 \\
0 & 1 & \cdots & 0 \\
\vdots & \vdots & \ddots & \vdots \\
0 & 0 & \cdots & 1
\end{pmatrix},
\end{equation}
we have the useful identity $\mathbf{B}_B = \mathbf{D}^T  \mathbf{B} \mathbf{D}$. By Sylvester's law of inertia~\cite[Theorem~4.5.8]{HornJohnson2013}, {the latter representation shows that} $\mathbf{B}_B = \mathbf{D}^T  \mathbf{B} \mathbf{D}$ is positive semidefinite. Therefore, the ferromagnetic Potts model on a regular graph with a uniform external field $B$ is also ferromagnetic in the sense of \Cref{d:ferro}. {A useful consequence is that, on regular graphs, the external field can be absorbed into the interaction matrix, allowing the model to be analyzed within the matrix-model framework. It is this crucial fact that sets regular graphs apart from non-regular ones.}
\end{example}

\begin{rem}[{Equivalence of factor and matrix models on regular graphs}]\label{remark:model_transform}
{Let $G=(V,E)$ be an arbitrary, finite $d$-regular graph. Then any factor model can be formulated as a matrix model, and vice versa:} 
Clearly, any matrix model {with interaction matrix $\mathbf{B}$} can be formulated as a factor model with specification components {$\psi(\sigma_i,\sigma_j) = B_{\sigma_i,\sigma_j}$ and} $\bar{\psi} \equiv 1$. Conversely, a factor model {with specification $\pmb{\psi}$} on $G$ can be formulated as a matrix model. To see this, we rewrite 
\begin{align*}
\mu_{\pmb{\psi}}({\bsigma}) &= \frac{1}{Z}\prod_{\{ i,j\}\in E}\psi(\sigma_i, \sigma_j) \prod_{i \in V} \bar{\psi}(\sigma_i) =  \frac{1}{Z}\prod_{\{ i,j\} \in E}\psi(\sigma_i, \sigma_j) \prod_{i\in V}\prod_{j \sim i}\bar{\psi}(\sigma_i)^{1/d}\\
&=  \frac{1}{Z}\prod_{\{ i,j\} \in E}\psi(\sigma_i, \sigma_j) \prod_{\{ i,j\}\in E}\left( \bar{\psi}(\sigma_i) \bar{\psi}(\sigma_j)\right)^{1/d}=  \frac{1}{Z}\prod_{\{ i,j\} \in E}\psi(\sigma_i, \sigma_j)\left( \bar{\psi}(\sigma_i) \bar{\psi}(\sigma_j)\right)^{1/d}.\\
\end{align*}
Define the interaction matrix $\mathbf{B}$ by
\[
B_{{\sigma_i},{\sigma_j}}
:=\psi(\sigma_i,\sigma_j)\bigl(\bar\psi(\sigma_i)\bar\psi(\sigma_j)\bigr)^{1/d}. 
\]
This transforms a factor model in a regular graph into a matrix model specified by an interaction matrix.
\end{rem}

\begin{definition}
 For a spin configuration $\sigma$, we refer to its unnormalized contribution $w_G({\bsigma})$ to the partition function of either a factor or matrix model as its \emph{Boltzmann weight}.   
\end{definition} 

{The previous remark shows that on regular graphs, we may use the factor model description and the interaction-matrix description of a spin system interchangeably. Although they
are essentially equivalent in our setting, we have decided to introduce both of them explicitly. }	 One reason is
that 
retaining both formulations makes it easier
to state the results {from the literature} in a form close to their original presentation, which
allows readers who are less familiar with the area to compare our discussion
directly with the source papers. A second reason is that each formulation is convenient for
a different part of our analysis: the factor model formulation is naturally
suited to the Bethe variational principle, whereas the interaction-matrix
formulation leads to the matrix-norm representation used in our second-moment
argument in \Cref{subsection:regular_compare}.

Keeping the previous discussion of the equivalence of factor and matrix models in mind, we formulate the following notions for spin models on \emph{random} graphs in the language of factor models. In the following, $G$ is assumed to be an arbitrary random graph.

\begin{definition}[Various measures associated to factor models {on random graphs}] \label{d:measures_random}
{Let $G = (V,E)$ be a finite random graph and $\pmb{\psi} =(\psi, \bar{\psi})$ be a factor model on $G$. }
	\begin{enumerate}
	\item \textbf{The random quenched measure} {on $\mathcal X^V$} is defined as 
	\begin{equation}
	\mu_G^{\mathrm{r\text{-}qu}}({\bsigma}) :=  \mu_{\pmb{\psi}}({\bsigma}), \qquad \bsigma\in \mathcal X^V.
	\end{equation}
	In particular, $\mu_G^{\mathrm{r\text{-}qu}}$ is a random variable on the underlying random graph model.
	\item \textbf{The averaged quenched measure} {on $\mathcal X^V$} is defined as
		\begin{equation}
		\mu_G^{\text{qu}}({\bsigma})  = \mathbb E\!\left[ {\mu_{\pmb{\psi}}({\bsigma})}
    \right], \qquad \bsigma\in \mathcal X^V.
		\end{equation}
        {In particular, $\mu_G^{\text{qu}}$ is a deterministic measure.}
		\item \textbf{The annealed measure}  {on $\mathcal X^V$} is defined as
		\begin{equation}
		 \mu_G^{\mathrm{an}}({\bsigma})
    := 
    {\frac{\mathbb E[\prod_{ \{i,j\}\in E}\psi(\sigma_i,\sigma_j) \prod_{i \in V}\bar{\psi}(\sigma_i) ]}{\mathbb E[Z_{G,\pmb{\psi}} ]}, \qquad \bsigma\in \mathcal X^V.}
		\end{equation}
        {In particular, $\mu_G^{\text{an}}$ is a deterministic measure.}
	\end{enumerate}
\end{definition}

The three measures introduced above correspond to different ways of averaging over the graph and spin configurations. For each measure, the logarithm of its normalizing constant is central to the thermodynamic description of the model, as it is directly related to the free energy and, through its derivatives, the relevant thermodynamic observables. Since the normalizing constant typically grows exponentially with the number of vertices, its logarithm is of order $|V|$. Normalizing the latter by $|V|$ therefore leads to the following three corresponding notions of pressure:


\begin{definition}[Pressure (per particle)/Free energy density] \label{d:ppp}
{Let $G = (V,E)$ be a finite random graph and $\pmb{\psi} =(\psi, \bar{\psi})$ be a factor model on $G$. }
	\begin{enumerate}
		\item \textbf{The random quenched pressure} is defined by 
		  \begin{equation}
			\varphi_G^{\mathrm{r\text{-}qu}}(\pmb{\psi}) = \frac{1}{|V|}\log Z_{G, \pmb{\psi}}.
		 \end{equation}
		\item \textbf{The averaged quenched pressure} is defined by
		\begin{equation}
			\varphi_G^{\text{qu}}(\pmb{\psi})= \frac{1}{|V|}\bbE[\log Z_{G, \pmb{\psi}}].
		 \end{equation}
		\item  \textbf{The annealed pressure }is defined by 
	\begin{equation}
	\varphi_G^{\text{an}}(\pmb{\psi}) = \frac{1}{|V|}\log \bbE[Z_{G, \pmb{\psi}}].
	\end{equation}
	\end{enumerate}
	For $V_n=[n]$, we {also} write $\varphi_n^\#$ for the corresponding finite-volume pressure, where $\#\in\{\mathrm{r\text{-}qu},\mathrm{qu},\mathrm{an}\}$.
\end{definition}

{While Definitions \ref{d:measures_random} and \ref{d:ppp} are formulated for a single random graph, we will be interested in the thermodynamic limit of a sequence of random graphs as the number of vertices tends to infinity. Thus, from now on,}
we consider factor models on a sequence of random graphs $(G_n)_{n\in\bbN}$ of diverging size.
{To emphasize this and for more transparent notation, we will replace all subscripts $G_n$ with $n$ and from now on, write $Z_n$ instead of $Z_{G_n}$.} 
We next define the relevant thermodynamic limits:
\begin{definition}[Thermodynamic limits]
Let $(G_n)_{n\in\bbN}$ be a sequence of random graphs $G_n = (V_n,E_n)$ with $|V_n|=n$ for all $n$, and $\pmb{\psi} =(\psi, \bar{\psi})$ be a factor model. As $n \to \infty$, the random quenched limit is given by
	\begin{equation}
		\varphi^{\mathrm{r\text{-}qu}}(\pmb\psi)
		:=\lim_{n\to\infty}\varphi_n^{\mathrm{r\text{-}qu}}(\pmb\psi),
	\end{equation}
	whenever this limit exists in the stated mode of convergence.
	For $\#\in\{\mathrm{qu},\mathrm{an}\}$, define the deterministic limit as $n \to \infty$ by
	\begin{equation}
		\varphi^\#(\pmb\psi)
		:=\lim_{n\to\infty}\varphi_n^\#(\pmb\psi),
	\end{equation}
	whenever 
    {it exists}.
\end{definition}
As {explained in the introduction, the focus of this article is on a collection of different representations for $\varphi^{\mathrm{qu}}(\pmb\psi)$ on random regular graphs, as well as the presentation of a relatively simple proof of the fact that $\varphi^{\mathrm{qu}}(\pmb\psi) = \varphi^{\mathrm{an}}(\pmb\psi)$ for ferromagnetic Potts models on these graphs.} 
To study the latter question, it is natural to compare the exponential growth rate of $\bbE[Z_n^2]$ with that of the square of $\bbE[Z_n]$.
We therefore define another central quantity, which captures the exponential growth rate of the second moment of the partition function as $n\to\infty$:

\begin{definition}[Normalized logarithm of the second moment] \label{d:second_moment}
{Let $(G_n)_{n\in\bbN}$ be a sequence of random graphs $G_n = (V_n,E_n)$ with $|V_n|=n$ for all $n$, and $\pmb{\psi} =(\psi, \bar{\psi})$ be a factor model. }
The thermodynamic limit of the normalized logarithm of the second moment of the partition function is
	 \begin{equation}
		\Psi_2(\pmb{\psi}) := \lim_{n\to\infty}\frac{1}{n}\log\bbE[Z_n^2],
	 \end{equation}
	 if the limit exists. 
\end{definition}

Unlike the other thermodynamic quantities, {$\Psi_2$} is not a common object discussed in the {statistical physics} literature. {However, we will see in \Cref{subsubsection:regular_second moment} that $Z_n^2$ can also be regarded} 
as the \textit{partition function} of a {suitably chosen} spin model on the random graph sequence $(G_n)_{n\in\bbN}$. Hence, $
   \Psi_2$ can also be thought of as the annealed pressure of a suitable physical model. 
Throughout the remainder of the paper, we
suppress the dependence on $\pmb{\psi}$ from the above notation {and assume that $\mathcal X = [q]$ for some $q \geq 2$.}


\subsection{Configuration Model}\label{s:configuration}

In this section, we define the specific sequence of random graphs that we will be working with. For this, we first define the more general configuration model. Let $\boldsymbol{d}=(d_i)_{i\in[n]}$ be a fixed sequence, where $d_i$ {represents} the degree of vertex $i\in V=[n]$. Denote  the total degree by $\ell_n = \sum_{i\in[n]} d_i$. We assume that $\ell_n$ is even. To construct the configuration model, we first set up a random matching of \textit{half-edges} from the given degree sequence. Initially, each vertex $v \in V$ is incident to $d_v$ half-edges, and we mark all these half-edges as \textit{unpaired}. We choose one half-edge $e$ arbitrarily from all the currently unpaired ones and choose a half-edge uniformly from the remaining unpaired half-edges to be paired with $e$. We then add an edge between the vertices to which these two half-edges are incident (they might be incident to the same vertex, in which case there is a self-loop), and mark these two half-edges as \textit{paired}. We repeat this pairing process until we run out of unpaired half-edges. The resulting random multi-graph is then called the \textit{configuration model} and is denoted by $\mathrm{CM}_n(\boldsymbol{d})$.

{The random $d$-regular graph and the uniform simple $d$-regular graph, the two main models of this article, can both be obtained from the configuration model.} 
The first one is given by the constant degree sequence $d_i=d$ for every $i\in[n]$, where $dn$ is even. The resulting random graph distribution is called the \textit{random $d$-regular graph} and is denoted by $\mathrm{CM}_n(d)$. {The uniform simple $d$-regular graph $\mathrm{UG}_n(d)$ can be obtained from $\mathrm{CM}_n(d)$ by conditioning $\mathrm{CM}_n(d)$ on being simple. }
{It is well known} that $\bbP(\mathrm{CM}_n(d) \text{ is simple}) \to c_d \in (0,1)$ for any $d \ge 2$. See \cite[Chapter 7]{vdH_2016}.

\subsection{Local convergence}\label{s:local_def}
 Local convergence describes the asymptotic structure of a graph as seen from a uniformly chosen vertex.
This notion is particularly relevant here, because a random $d$-regular graph is sparse and, at every fixed radius, locally resembles the infinite $d$-regular tree.
Results reviewed later relate the quenched pressure of spin models on such graphs to the belief-propagation recursion and the Bethe free energy on this limiting tree. See \Cref{d:BP,d:Bethe_free} below.

To formulate local convergence, let $(\mathbb T,o)$ be a random rooted tree and, for $t\in\bbN$, write
$
\mathbb T_t:=B^{\sss(\mathbb T)}(o;t)
$
for the radius-$t$ neighborhood of the root, regarded as a rooted graph.
For rooted graphs $H_1$ and $H_2$, we write $H_1\simeq H_2$ if they are isomorphic as rooted graphs. 
    
\begin{definition}[Locally tree-like graphs]\label{d:locally_tree}
	Let $G_n = (V_n ,E_n)$ be a sequence of random graphs such that $|V_n| \to \infty$, and let $(\mathbb{T},o)$ be a random rooted tree.
	For each $n$, let $o_n$ be a uniformly chosen vertex from $V_n$, independently of $G_n$.
	We say that $(G_n)$ \emph{converges locally in probability} to $(\mathbb T,o)$ if, for every $t\in\bbN$ and every finite rooted graph $H$,
	\begin{equation}
		\bbP\!\left(B^{\sss(G_n)}(o_n;t)\simeq H\mid G_n\right)
		\overset{\bbP}{\longrightarrow}
		\bbP(\mathbb T_t\simeq H).
	\end{equation}
	If $(G_n)$ converges locally in probability to {a random rooted tree,}
    then we say that $(G_n)$ is \emph{locally tree-like}, and if further the local limit is a regular tree, then we say that it is \emph{locally regular tree-like}.
\end{definition}

\subsection{{A second-moment criterion for equality of quenched and annealed pressures}}
 
 {We next explain how second-moment information can be used to prove equality of quenched and annealed pressures. The key quantity is the exponential growth rate $\Psi_2$ of the second moment of the partition function from \Cref{d:second_moment}. The following proposition shows that, together with concentration of the quenched pressure, the identity $\Psi_2=2\varphi^{\rm an}$ yields a general criterion for quenched--annealed agreement. The result is sometimes referred to as \emph{Frieze's trick}:}
 

\begin{proposition} [Frieze's trick \cite{kavita}] \label{proposition:Frieze}
 {Let $(G_n)_{n\in\bbN}$ be a sequence of random graphs $G_n = (V_n,E_n)$ defined on the same probability space with $|V_n|=n$ for all $n$, and $\pmb{\psi} =(\psi, \bar{\psi})$ be a factor model. }
  Assume that both $\varphi^{\mathrm{an}}$ and $\Psi_2$ exist, {with}
  \begin{equation} \label{e:second_moment_condition}
	2\varphi^{\mathrm{an}}=\Psi_2.
  \end{equation}
  Suppose further that the random quenched pressure is \emph{concentrated}, i.e., for any $\delta>0$, there exists a constant $c=c(\delta) >0$ such that 
  \begin{equation}\label{e:concentration}
    \bbP \left( \left|\frac{1}{n}\log Z_n -  \bbE\left[ \frac{1}{n}\log Z_n\right] \right| >\delta \right) \leq \e^{-c n}.
  \end{equation}
   Then the random quenched and quenched limits exist and both are equal to the annealed limit. More specifically, a.s.,
  $\varphi^{\mathrm{an}} = \varphi^{\mathrm{r\text{-}qu}} = \varphi^{\mathrm{qu}}$.
\end{proposition}
A proof of Frieze's trick is given in \cite{kavita}. 
{The concentration condition \eqref{e:concentration} can be established by relatively standard martingale techniques for a wide range of random graph and spin models. The more challenging ingredient is the second-moment condition \eqref{e:second_moment_condition}. Indeed, there is no a priori reason for the exponential growth rate of the second moment, \(\Psi_2\), to coincide with twice the annealed pressure. Establishing this identity for ferromagnetic Potts models on random regular graphs is one of the main objectives of the paper, and we return to it in \Cref{subsection:regular_compare}. As a warm-up, we begin by verifying \eqref{e:concentration} for the Potts model on the configuration model.}

 
 \begin{example}[Potts models on the configuration model satisfy \eqref{e:concentration}] \label{ex:Potts_con}
 A standard martingale argument shows that the $q$-state Potts model with parameters $(\beta, B)$ on a configuration model satisfies \eqref{e:concentration}.
Let $\mathrm{CM}_n(\boldsymbol{d})$ be a general configuration model as defined in \Cref{s:configuration}, and denote by $\ell_n$ the total degree. Denote by $(\calF_t)_{0 \leq t \le \ell_n/2}$ the filtration generated by the edge-pairing process of $\mathrm{CM}_n(\boldsymbol{d})$ described in \Cref{s:configuration}. We then show that the Doob martingale $(\bbE[\log Z_n \mid  \calF_t])_{0 \leq t \le \ell_n/2}$ has the property of bounded differences. An application of the Azuma--Hoeffding inequality will then establish the desired concentration.

{We label the half-edges arbitrarily and let $t \in [\ell_n/2]$. To establish the property of bounded differences, we} couple two configuration models $G \sim \mathrm{CM}_n(\boldsymbol{d})$ and $G' \sim \mathrm{CM}_n(\boldsymbol{d})$ such that $G'$ is obtained from $G$ by a single edge switching of the $t$-th paired half-edge. 
More explicitly, sample $t-1$ pairings of the configuration model $G$ as prescribed by $
\calF_{t-1}$, and use the same pairings for $G'$, so the first $t-1$ pairings of $G$ and $G'$ are exactly the same. At step $t$, we pick {the unpaired half-edge $x_t$ with the lowest label} in $G$ (and hence in $G'$) as the next half-edge to be paired in both $G$ and $G'$. For $G$, pair it uniformly with some other unpaired half-edge $y_t$ in $G$, as specified by $
\calF_t$. For $G'$, we also draw uniformly a half-edge $y_t'$ from the remaining unpaired half-edges in $G'$ independently of {the choice of $y_t$ and pair $x_t$ with $y_t'$}.
{Then continue to pair half-edges in $G$ and $G'$ identically until $y_t'$ is paired in $G$ (say, to $x_t'$). At this step, pair $y_t$ and $x_t'$ in $G'$. Afterwards, half-edges can again be paired identically in $G$ and $G'$ until all half-edges are paired. } 
The final graphs $G$ and $G'$ are configuration models and differ in at most {two edges}. This completes the coupling.  

We now show the bounded differences of the log-partition function. For a fixed $t \le \ell_n/2$, let $G$ and $G'$ be coupled as above and denote by $Z_n$ the Potts partition function on $G$ and by $Z_n'$ the one on $G'$ at $\beta, B \ge 0$. By construction, $Z_n$ given $\calF_{t-1}$ has the same distribution as $Z_n'$ given $\calF_{t}$. Thus 
\begin{align*}
	\big|\bbE[ \log Z_n \mid \calF_{t}] - \bbE[\log Z_n \mid \calF_{t-1} ] \big|  &= \big|\bbE[ \log Z_n \mid \calF_{t}] - \bbE[\log Z_n' \mid \calF_{t} ] \big|\\
	&\le \bbE[ |\log Z_n - \log Z_n' | \mid \calF_{t}] .
\end{align*}
To bound $|\log Z_n - \log Z_n' |$, note that the above coupled graphs differ by a single switching of two edges.
For any spin configuration, the switching {can change the number of monochromatic edges by at most $2$. Thus, its }
Boltzmann weight is rescaled by a factor $r \in [\e^{-2\beta},\e^{2\beta}]$.
Consequently, 
\[
|\log Z_n - \log Z_n' | \le {2} \beta \quad  a.s.
\]
Applying the Azuma-Hoeffding inequality 
\cite[Theorem 3.2.1]{Roch_2024} to the Doob martingale $( \bbE[ \log Z_n \mid \calF_{t}])_{t \le [\ell_n /2]}$, we obtain \eqref{e:concentration}. 
 \end{example}

{Similarly, one can use Doob martingales and the Azuma-Hoeffding inequality to prove concentration of the quenched pressure per particle for other spin systems on random graph models, such as } 
those with independent edges, e.g., Erd\H{o}s-R\'enyi graphs $G(n,c/n)$, in which each pair of vertices is {independently} connected by an edge with probability $c/n$.

While Frieze's trick provides a sufficient condition for quenched--annealed agreement, it is natural to ask to what extent its hypotheses are also necessary. A complete converse is false in general. Nevertheless, the next proposition identifies a mechanism that forces both the second-moment condition and quenched--annealed agreement to fail. More precisely, if the annealed pressure is dominated by an \emph{exponentially rare} graph event, then the two conditions fail simultaneously:
\begin{proposition}[Condition for quenched being different from annealed]
	 {Let $(G_n)_{n\in\bbN}$ be a sequence of random graphs $G_n = (V_n,E_n)$ with $|V_n|=n$ for all $n$, and $\pmb{\psi} =(\psi, \bar{\psi})$ be a factor model.}
Suppose there exists a sequence $(\calE_n)_n$ of events satisfying $\bbP(\calE_n)  = \e^{-\Theta(n)}$ and {$\bbE[Z_n] =(1+ \e^{-\Theta(n)})\bbE[Z_n \mathds{1}_{\calE_n}]$}. Then a.s., 
\begin{equation}
 \liminf_{n\to\infty}( \varphi^{\mathrm{an}}_n-\varphi_n^{\mathrm{r\text{-}qu}})>0 \quad \text{and  } \quad \liminf_{n\to\infty}\left(\frac{1}{n}\log\bbE[Z_n^2]-\frac{2}{n}\log\bbE[Z_n]
\right) >0.
\end{equation}
If further the relevant thermodynamic limits exist, then
	\[
	\varphi^{\mathrm{r\text{-}qu}} < \varphi^{\mathrm{an}} \quad \text{and} \quad \Psi_2 > 2\varphi^{\mathrm{an}}.
	\] 
\end{proposition}
\begin{proof}We decompose the log-partition function into two parts using $\calE_n$, 
	\[
	\log Z_n  = (\log Z_n) \mathds{1}_{\calE_n} + (\log Z_{n })\mathds{1}_{\calE_n^c}.
	\]
	By the Borel-Cantelli Lemma {applied} to $(\calE_n)_{n\in\bbN}$, the first {summand} vanishes as $n$ tends to infinity,
	\begin{align}\label{eq_event0}
	    \log Z_n \mathds{1}_{\calE_n}  \overset{a.s.}{\longrightarrow} 0.
	\end{align}
	For the second {summand}, for any $\epsilon >0$, {Markov's inequality implies that}
	\[
	 \bbP( \log Z_n \mathds{1}_{\calE_n^c} > n(\varphi^{\text{an}}_n -\ve )) \le \bbP( Z_n \mathds{1}_{\calE_n^c}  >\e^{n (\varphi^{\text{an}}_n -\ve )}) \le \e^{-n ( \varphi^{\text{an} }_n-\ve )} \bbE[ Z_n \mathds{1}_{\calE_n^c}].
	\]
	By assumption, 
	\[
	\bbE[Z_n\mathds{1}_{\calE_n^c}]
	= \e^{-\Theta(n)}\bbE[Z_n]
	= \e^{-\Theta(n)}\e^{n\varphi^{\mathrm{an}}_n}.
	\]
	Thus, 
	\begin{align} \label{eq_event1}
	    \bbP( \log Z_n \mathds{1}_{\calE_n^c} > n(\varphi^{\text{an}}_n -\ve )) \le \e^{n\ve }\e^{-\Theta(n)}. 
	\end{align}
We now take $\ve >0$ small enough that the right-hand side {of \eqref{eq_event1}} is summable in $n$. Then, by the Borel-Cantelli lemma, a.s.
\begin{align}\label{eq_event2}
    \log Z_n \mathds{1}_{\calE_n^c} \le n( \varphi_n^{\text{an}} -\ve ) 
\end{align}
for all but finitely many $n$. Combining {\eqref{eq_event0} and \eqref{eq_event2}}, we conclude that a.s.,
\[
 \liminf_{n\to\infty}( \varphi^{\mathrm{an}}_n-\varphi_n^{\mathrm{r\text{-}qu}}) \ge \ve.
\]
Next, we prove the second part.
By the Cauchy--Schwarz inequality {and the assumption},
\[
 \bbE[Z_n^2 ] \ge \frac{(\bbE[ Z_n \mathds{1}_{\calE_n} ])^2 }{\bbP(\calE_n) }  =\frac{ {(1+ \e^{-\Theta(n)})} \bbE[Z_n]^2}{\e^{-\Theta(n)}} .
\]
Therefore, 
\[
\frac{1}{n}\log\bbE[Z_n^2]
\ge \Theta(1)+\frac{2}{n}\log\bbE[Z_n].
\]
Taking the limit $n\to\infty$ we obtain that $\Psi_2 > 2\varphi^{\mathrm{an}}$, completing the proof.  
\end{proof}

The proposition shows that the identity \( \Psi_2=2\varphi^{\mathrm{an}} \) rules out a simple mechanism by which quenched and annealed quantities can differ, namely that the annealed pressure is dominated by an exponentially rare graph event. In this sense, $\Psi_2$ provides partial information about whether the annealed partition function is representative of typical graph realizations.
This raises the question  whether an exponentially larger second
moment necessarily implies a discrepancy between the quenched and annealed
pressures. More precisely, assuming that the relevant thermodynamic limits
exist, does the implication
\[
    \Psi_2>2\varphi^{\mathrm{an}}
    \quad\Longrightarrow\quad
    \varphi^{\mathrm{qu}}<\varphi^{\mathrm{an}}
\]
hold under some additional assumptions on the spin models we are considering?
An affirmative answer would show that the second-moment condition
$\Psi_2=2\varphi^{\mathrm{an}}$, together with concentration, not only provides a sufficient criterion for
quenched-annealed agreement, but also captures
a necessary feature of this agreement.

\subsection{Related models}
{Results similar to those presented here have recently been obtained for the ferromagnetic Ising model on random $d$-regular graphs \emph{at fixed magnetization} in \cite{gheissari2025fixed}: The authors show that in the regime below the tree reconstruction threshold, the quenched free-energy density converges to the annealed free energy density. Moreover, the limiting value coincides with the Bethe free energy associated with a suitable Ising measure on the infinite $d$-regular tree.}

Another closely related model is the \emph{antiferromagnetic} Potts model, which corresponds to \eqref{eq_def_Potts} with $\beta < 0$. This model penalizes monochromatic edges, and notably, at zero temperature ($\beta =- \infty$), its partition function coincides with the number of proper $q$-colorings of the underlying graph. Salez~\cite{Salez2016} used interpolation to establish the existence of {an almost sure limit for the quenched} pressure of antiferromagnetic Ising and Potts models on configuration model graphs. The antiferromagnetic Potts model on the Poissonian Erd\H{o}s--R\'{e}nyi random graph was studied by Contucci, Dommers, Giardin\`a and Starr~\cite{ContucciDommersGiardinaStarr2013}, who developed an interpolation-based variational framework and proved the existence of a phase transition at negative temperature. 
However, the behavior of the antiferromagnetic Ising model is generally quite different from the ferromagnetic one. For example, here, the limit of the quenched pressure actually depends on the exact sequence of random graphs. Even the quenched-pressure limits of two random graph models with the same local limit may be different, as observed in \cite{Sun}. For general background on the antiferromagnetic Potts model, see \cite{Sokal2005}.

For a broader treatment of spin models on random graphs, see~\cite[Part IV]{vdH_SPoRG}.

\section{Results for random regular graphs}\label{s:d_reg}

In this section, we study the first and second moments of partition functions for spin models on random $d$-regular graphs.
In \Cref{subsection:regular_moment}, we review several representations of the annealed pressure and obtain corresponding representations of the normalized logarithm $\Psi_2$ of the second moment.
Finally, in
\Cref{subsection:regular_compare}, we use these representations to establish the relation $\Psi_2 = 2\varphi^{\mathrm{an}}$,
derive a second-moment proof of quenched--annealed agreement for the configuration model, and transfer these conclusions to the uniform simple $d$-regular graph.

\subsection{Representations of the moments of the partition function}\label{subsection:regular_moment}
The aim of this section is to develop representations of the normalized logarithms of the first and second
moments of the partition function that will later allow us to compare their
exponential growth rates. We begin in
\Cref{subsubsection:regular_first_moment} by reviewing several variational
representations of the annealed pressure on random regular graphs. In
\Cref{subsubsection:regular_second moment}, we extend these first-moment
representations to the second moment by interpreting $Z_n^2$ as the partition
function of a suitable spin model. {Recall that for notational convenience, we assume that $\mathcal X = [q]$ for some $q \geq 2$.}

\subsubsection{First moment}\label{subsubsection:regular_first_moment}

Table \ref{table-one} presents an overview of the representations discussed in this article.

We present two
representations of the annealed pressure of {spin} 
models on random $d$-regular graphs: the Bethe variational formula of Dembo, Montanari, Sly and
Sun~\cite{Sun} {and} the matrix-norm representation of Galanis,
\v{S}tefankovi\v{c}, Vigoda and Yang~\cite{BIS}. Two further representations specific to the Ising model can be found in \Cref{s:ising models}: the
moment-generating-function representation of Can~\cite{can}, and the edge-bicolored graph approach~\cite{BorinskyRenWiesmann}.
Although these representations characterize the same thermodynamic quantity with slight differences in how generally the results apply, the connections between their formulations are not immediate, and each may be useful in a different context.
For the reader's convenience, we describe the main ideas behind their derivations and
clarify how their different parametrizations are related. We treat the
matrix-norm representation in greater detail because it will later provide
the main tool for comparing the first and second moments of the partition
function.

Our overview is not intended to be exhaustive.
We highlight only the results most directly relevant to the questions on random regular graphs considered in this paper, and refer to the works cited below and the references therein for a broader account of the literature.
\medskip

\paragraph{\bf Variational formula via Bethe recursion.}
Our first result on the thermodynamic limit of the annealed pressure for random regular graphs is due to Dembo, Montanari, Sly and Sun \cite{Sun}. We first introduce the empirical edge distribution, which is both the variational parameter in their representation and a key object in its proof.

\begin{table}[t]
    \centering
    \footnotesize
    \setlength{\tabcolsep}{4pt}
    \renewcommand{\arraystretch}{1}

    \caption{{Representations of the annealed pressure}}
    \label{tab:first-moment-representations}

    \begin{tabularx}{\textwidth}{
        @{}
        L{0.19\textwidth}
        L{0.27\textwidth}
        L{0.19\textwidth}
        Y
        @{}
    }
        \toprule
        Representation
        & Variational parameter 
        & Applicable models 
        & Distinctive feature
        \\
        \midrule

        Bethe 
        representation,
        \Cref{proposition:rep_1}
        &
        Empirical edge distribution \(\mathbf{x}=(x_{ij})_{i,j \in [q]}\) 
        &
        Factor models with permissive $\boldsymbol{\psi}$ on random
        \(d\)-regular graphs, \(d\geq2\).
        &
        Most general representation considered here. It expresses the pressure as an entropy--energy optimization over local empirical statistics and provides a rigorous version of the Bethe prediction.
       \\  

        \addlinespace[0.5ex]

        Matrix-norm representation,
        \Cref{t:rep_3}
        &
        Empirical spin distribution \(\boldsymbol{\alpha}\), expressed through projective coordinates \(\boldsymbol{u}\in[0,\infty)^{q}\) via \( \alpha_i=\frac{u_i^p}{\sum_j u_j^p}. \)
        &
        Ergodic ferromagnetic matrix models on random
        \(d\)-regular graphs, \(d\geq 2\).
        &
        Recasts the variational problem as an induced matrix-norm formula. Its algebraic structure is the basis of our second-moment argument.
        \\
	\\[1ex]
 \midrule
        Matching representation,
        Theorem~\ref{t:can_rep}
        &
        Empirical proportion of positive spins \(t\in[0,1]\)
        &
        Ising model on random \(d\)-regular graphs,
        \(d\geq2\).
        &
        Scalar optimization over the positive-spin density, obtained by analyzing matched half-edges; useful for explicit analysis and valid also for \(d=2\).
        \\

        \addlinespace[0.2ex]

        Edge-bicolored-graph representation,
        \Cref{t:combinatorial_Ising_rep}
        &
        \(\theta\in\bbR\), angular parametrization 
        of the two edge-color generating variables 
        &
        Ising model on random \(d\)-regular graphs,
        \(d\geq3\).
        &
        Turns the high-temperature expansion into a weighted enumeration of edge-bicolored graphs and reduces the pressure to a one-dimensional angular optimization.
        \\
        \bottomrule
    \end{tabularx}
    \label{table-one}
\end{table}
\FloatBarrier


\begin{definition}[Empirical edge distribution] \label[definition]{d:edge_color}
Fix a spin configuration $\bsigma\in [q]^V$ on a finite graph $G= (V,E)$. Denote by $e(i,j)$ the number of edges connecting vertices of {spin} $i$ and $j$. The \textit{empirical edge distribution} is defined to be the probability matrix $\mathbf{x}=(x_{ij})_{i,j\in[q]} \in \calP_{q \times q}$ with entries 
	\begin{equation}
		x_{ij} = 
        \begin{cases}
            \frac{e(i,j)}{2|E|}, & \text{ if  } i\not=j, \\
            \frac{e(i,i)}{|E|}, & \text{else}.
        \end{cases}
	\end{equation}
	Further, the one-point marginal 
    vector $\boldsymbol{\alpha} = \boldsymbol{\alpha}(\mathbf{x}) = (\alpha_1, \ldots, \alpha_q) \in \calP_q$ {associated with $\mathbf{x}$} is given by
\begin{equation}\label{e:marginal}
	\alpha_i = \sum_{j=1}^q x_{ij}, \qquad i \in [q].
\end{equation}
\end{definition}

 Thus, each component $x_{ij}$ of the empirical edge distribution matrix $\mathbf{x}$ is half the proportion of edges between vertices of color $i$ and $j$ whenever $i \ne j$, whereas, on the diagonal, each $x_{ii}$ is the proportion of edges between vertices of color $i$. By construction, $\mathbf{x}$ is symmetric. Note also that if the graph $G$ is $d$-regular, then
 \begin{align}\label{eq_marg_reg}
      \alpha_i =\sum_{j=1}^q x_{ij} = \frac{1}{|E|} \left(e(i,i) + \sum_{j \ne i}    \frac{e(i,j)}{2}\right) = \frac{d\cdot | \{ v \in V: \sigma_v = i\}| }{2|E|} = \frac{|\{ v \in V: \sigma_v = i\}|}{|V|}.
 \end{align}
Thus, we see that $\alpha_i$ is the proportion of vertices of color $i$ in the graph, and in this case we call $\boldsymbol{\alpha}$ the \textit{empirical spin distribution}.

For functions $f:{[q]}^2\to\bbR$ and $\bar f:{[q]}\to\bbR$, {we next define their averages with respect to $\mathbf{x}, \boldsymbol{\alpha}$ as}
\[
\langle f\rangle_{\mathbf{x}}
:=\sum_{i,j\in {[q]}}x_{ij}f(i,j),
\qquad
\langle\bar f\rangle_{\boldsymbol{\alpha}}
:=\sum_{i\in {[q]}}\alpha_i\bar f(i).
\]
With this notation, we introduce the following quantity:
\begin{definition}[{Bethe rate function \cite[(9)]{Sun}}]\label{d:Beth_polytope}
 For a general factor model on a $d$-regular graph, the \textit{Bethe rate function} is defined by 
\begin{equation}\label{e:Bethe_local}
	\Psi_{\mathrm{Bethe}}(\mathbf{x}) := \langle \log \bar{\psi} \rangle_{\boldsymbol{\alpha}} - (d-1)H(\boldsymbol{\alpha}) +(d/2) [ \langle \log \psi\rangle_{\mathbf{x}}  +H(\mathbf{x})], \quad \mathbf{x} \in \calP_{q\times q}.
\end{equation}
\end{definition}

The Bethe rate function is motivated by the Bethe approximation in statistical physics on finite graphs; see, e.g., \cite[(37) and (38)]{BP}, where it is called the \textit{constrained Bethe free energy}.
For the extension of {\eqref{e:Bethe_local}} 
to more general locally tree-like graphs, see \cite[(1.23)]{Dembo2}.  

\begin{example}[Bethe rate function for Potts models]
{Recall the specific form of $\psi, \bar{\psi}$ from \eqref{def_spec_Potts} for the} $q$-state Potts model on a $d$-regular graph. {Plugging this into \eqref{e:Bethe_local}} yields 
\begin{align*}
	\Psi_{\mathrm{Bethe}} (\mathbf{x}) 
	&= B\alpha_1 - (d-1)H(\boldsymbol{\alpha}) +\frac{d}{2}\left( \beta \sum_{i \in [q]} x_{ii} +H(\mathbf{x})\right) .  
\end{align*}
\end{example}

 With \eqref{e:Bethe_local}, we may now state the first representation of the annealed pressure per particle as an optimization problem over the potential empirical edge distributions:
 
\begin{proposition}[{\cite[Theorem 3]{Sun}}] \label{proposition:rep_1}
	{Fix $d \ge 2$.} For any factor model {as in Definition~\ref{d:factor}} specified by $\pmb{\psi} = (\psi,\bar{\psi})$  
    on a sequence of random $d$-regular graphs generated by the configuration model,  the thermodynamic limit of the annealed pressure exists and  is given by
\begin{align}\label{equation:Bethe_prediction}
  \varphi^{\mathrm{an}}
  =\Psi^*(d,\pmb\psi)
  :=\sup_{\mathbf{x}\in\calP_{q\times q}}\Psi_{\mathrm{Bethe}}(\mathbf{x}).
\end{align}
    Moreover, the maximizer is an interior point of $\calP_{q\times q}$. 
\end{proposition}

\begin{proof}[Outline of the proof of \Cref{proposition:rep_1}] 
Let $G_n= (V_n,E_n)$ be a random $d$-regular graph generated by the configuration model with $|V_n| = n$, and $\pmb\psi$ be a factor model on $G_n$ with partition function $Z_n$. The basic idea of the proof is to decompose the partition function into a polynomial number of contributions from different empirical edge distributions. This is possible since all configurations inducing the same empirical edge distribution contribute the same Boltzmann weight to the partition function. Consequently, the partition function can be reorganized by grouping configurations according to their empirical edge distributions, thereby reducing a sum over all configurations to a sum over empirical edge distributions.

 Fix an edge distribution $\mathbf{x} \in \calP_{q \times q}$ and consider any configuration $\sigma {\in {[q]}^V}$ that induces the edge distribution $\mathbf{x}$. Then, {using \eqref{eq_marg_reg},} the $\mu_{\pmb{\psi}}$ and $Z_{{n},\pmb{\psi} }$ from \Cref{d:factor} satisfy
\begin{equation}\label{eq:Bolt_weight}
	   \mu_{\pmb{\psi}}({\bsigma}) \cdot Z_{{n},\pmb{\psi}} = \prod_{\{u,v\}\in E}\psi(\sigma_u,\sigma_v ) \prod_{v \in V} \bar{\psi}(\sigma_v) = \prod_{i,j \in {[q]}} \psi(i,j)^{|E|x_{ij} }\prod_{i \in {[q]}} \bar{\psi}(i)^{|V|\alpha_i}.
\end{equation}
With this notation, {and using that $|E_n| = dn/2$,} we may write 
\begin{align}\label{e:weight_edge}
 \mu_{\pmb{\psi}}({\bsigma}) \cdot Z_{{n},\pmb{\psi}} &=  \prod_{i,j \in{[q]} } \psi(i,j)^{dnx_{ij}/2 }\prod_{i  \in{[q]} } \bar{\psi}(i)^{n\alpha_i} = \e^{(dn/2)\sum_{i,j} x_{ij}\log \psi(i,j)+ n\sum_i   \alpha_i \log \bar{\psi}(i)} \nonumber \\
& = \e^{(dn/2) \langle \log \psi \rangle_{\mathbf{x}} + n \langle \log \bar{\psi} \rangle_{\boldsymbol{\alpha}}}.
\end{align}
%
{According to \eqref{e:weight_edge}, every configuration with the same empirical edge distribution renders the same deterministic contribution to the partition function.} We may thus write $Z_n = \sum_{\mathbf{x}} Z_n(\mathbf{x})$, where each $Z_n(\mathbf{x})$ is the contribution of the weights from empirical edge distribution $\mathbf{x}$ {to the partition function.}  Taking expectation over $G_n$ yields
\[
\bbE[Z_n(\mathbf{x})]  =   \e^{(dn/2) \langle \log \psi \rangle_{\mathbf{x}} + n \langle \log \bar{\psi} \rangle_{\boldsymbol{\alpha}}} \bbE[|\{\bsigma\in \Omega: {\mathbf{x}({\bsigma}) =} \mathbf{x}\}|],
\]
where {we write} ${\mathbf{x}({\bsigma}) = \mathbf{x}}$ if and only if $\bsigma $ induces the edge distribution $\mathbf{x}$. It remains to account for the number {$|\bsigma\in \Omega: {\mathbf{x}({\bsigma}) = \mathbf{x}}\}|$} of configurations leading to a certain empirical edge distribution. Note that {this number} is a random variable, since the empirical edge distribution of a fixed spin configuration depends on the exact incidence structure of the {random} graph. For a random $d$-regular graph, the latter expectation can be computed directly through counting, and an application of Stirling's formula gives (see \cite[{proof of } Theorem 3]{Sun})
\[
\bbE[|\{\bsigma\in \Omega:{\mathbf{x}({\bsigma}) = \mathbf{x}} \}|] = n^{O(1)} \e^{n [ (d/2) H(\mathbf{x}) - (d-1)H(\boldsymbol{\alpha})]}.
\]
 Since there are only polynomially (w.r.t. the graph size $n$) many admissible edge distributions, we can apply the Laplace principle to express the annealed pressure as the maximum of $\Psi_{\mathrm{Bethe}}(\mathbf{x})$ over the possible \emph{realizable} distributions $\mathbf{x}$. We now  argue that the global maximum of $\Psi_{\mathrm{Bethe}}$ lies in the interior of $\calP_{q \times q}$ (see also \cite[Theorem 1.16]{Dembo2}).

Suppose, toward a contradiction, that $\mathbf{x}$ is a local maximum of $\Psi_{\mathrm{Bethe}}$ and has a coordinate equal to zero. Set $I:= \{ (i,j) : x_{ij} = 0\}$ (notice that this set is symmetric). Since $\mathbf{x}$ is a probability distribution, there exists an $x_{ab} >0$. We consider a smooth curve $\mathbf{x}(t)$ that starts from $\mathbf{x}$ at $t=0$ and moves into the interior of $\calP_{q\times q}$, i.e. a curve $\mathbf{x}: [0,t_0]\to \calP_{q\times q}$ such that $x_{ij}(t) >0$ for all $i,j\in [q] $ and $t>0$. When $a\ne b$, we define
\begin{equation}
	x_{ij}(t)=\begin{cases} t, & (i,j) \in I,\\
			x_{ij}, & (i,j) \not\in I,  \{i,j\} \ne \{a,b\},\\
			x_{ab}- t|I|/2, &\{i,j\} = \{a,b\}.
	\end{cases}
\end{equation}
Let $\alpha_i(t):=\sum_j x_{ij}(t)$.
By the positivity of $\psi$ and $\bar{\psi}$, differentiating \eqref{e:Bethe_local} gives, for all sufficiently small $t>0$,
\begin{align*}
\frac{d}{dt}\Psi_{\mathrm{Bethe}}(\mathbf{x}(t))
&= \sum_i \bigl(\log\bar{\psi}(i)+(d-1)\log\alpha_i(t)\bigr)\alpha_i'(t) +\frac d2\sum_{i,j}\bigl(\log\psi(i,j)-\log x_{ij}(t)\bigr)x_{ij}'(t),
\end{align*}
which diverges to $+\infty$ as $t \downarrow 0$. 
Applying the mean value theorem to $\Psi_{\mathrm{Bethe}}(\mathbf{x}(t))$ on some small interval $[0,\epsilon]$ gives $\Psi_{\mathrm{Bethe}}(\mathbf{x}(\epsilon)) > \Psi_{\mathrm{Bethe}}(\mathbf{x})$, which contradicts the assumption that $\mathbf{x}$ is a local maximum. When $a =b$, $\mathbf{x}$ is defined exactly the same, except for $x_{ab}(t) = x_{ab}- |I|t$.
 

Finally, since the global maximum of $\Psi_{\mathrm{Bethe}}$ lies in the interior of $\calP_{q \times q}$, it is sufficiently close to one of the realizable $\mathbf{x}$, which proves \eqref{equation:Bethe_prediction}.
\end{proof}

The proof of \Cref{proposition:rep_1} shows that the value of the Bethe rate function {$\Psi_{\mathrm{Bethe}}$} evaluated at an empirical edge distribution $\mathbf{x}$ characterizes the growth rate of the expectation $\bbE[ Z_n(\mathbf{x})]$ by
\begin{equation}\label{e:Bethe_repres}
\bbE[Z_n(\mathbf{x})] = \e^{n \Psi_{\mathrm{Bethe}}(\mathbf{x}) +o(n) }.
\end{equation}
Thus, in the case of $d$-regular graphs generated by the configuration model, the generally complicated variational characterization \eqref{equation:Bethe_prediction} has a simple and rigorous probabilistic interpretation.

\begin{remark}
    Dembo, Montanari, Sly and Sun~\cite{Sun} establish the annealed pressure of general factor models on the configuration model.
    {For ferromagnetic Potts models, their formula coincides with a previously established lower bound on the quenched pressure for locally regular, tree-like graph sequences \cite{Dembo2}. This coincidence identifies the quenched pressure and establishes quenched-annealed agreement for ferromagnetic Potts models on random regular graphs generated by the configuration model. More generally, their results apply to factor models on arbitrary graph sequences that converge locally to the infinite $d$-regular tree satisfying a set of conditions.}
\end{remark}




\paragraph{\bf Variational formula via matrix norm.} We now turn to a second representation of the annealed pressure, formulated for the interaction-matrix models of \Cref{d:ferro} rather than in the factor-model formalism used above. This representation is obtained by decomposing the partition function according to the empirical spin distribution. We discuss it in greater detail because its matrix-norm formulation will be used later to analyze the second moment of the partition function.  

Consider a ferromagnetic matrix model $\mathbf{B}$ with $q$ {spins} and,
following \cite[(15)]{BIS}, define $\Phi: [0,\infty)^q \backslash \{0\} \to \bbR$ by 

\begin{equation} \label{e:magic}
	\Phi(\boldsymbol{u}) := \frac{d}{2}\log \left( \sum_{i,j=1}^q B_{i,j} u_iu_j\right) - (d-1) \log \left( \sum_{i=1}^q u_i^{\frac{d}{d-1}}\right).
\end{equation}
In the notation, we suppress the dependence of $\Phi$ on $d$ and $\mathbf{B}$. {The function $\Phi$ in \eqref{e:magic} is invariant under scaling, i.e., $\Phi(\lambda \boldsymbol{u}) = \Phi(\boldsymbol{u})$ for any $\lambda >0$. Therefore, we will also consider $\Phi$ as a function defined on the equivalence classes of the half-rays in $[0,\infty)^q$ and swap between these two definitions according to whichever is more convenient.}

\begin{theorem}[{\cite[Lemmas 11 and  12]{BIS}}]\label{t:rep_3} Let $d \ge 2$. For any ferromagnetic model on a {sequence of} random $d$-regular graphs with interaction matrix $\mathbf{B}$, the annealed pressure {exists and} is given by
	\begin{equation}
		\varphi^{\mathrm{an}} = \max_{\boldsymbol{u} \in [0,\infty)^q \backslash \{0\}} \Phi(\boldsymbol{u}).
	\end{equation}
	In particular, this leads to the matrix-norm representation of the annealed pressure as
	\begin{equation}\label{e:annealed_norm}
		\varphi^{\mathrm{an}} =  d\log || \hat{\mathbf{B}}||_{\frac{d}{d-1}\to 2},
	\end{equation}
    where $\hat{\mathbf{B}}$ satisfies 
	\[
	  \mathbf{B} = \hat{\mathbf{B}}^T\hat{\mathbf{B}}.
	\]
\end{theorem}
We next outline how the matrix-norm representation of Theorem \ref{t:rep_3} is derived in \cite{BIS}. To state the {key} lemma {on which Theorem \ref{t:rep_3} is based}, 
we define $\varphi: \calP_q \to \bbR$ \cite[(13)]{BIS}, 
\begin{equation}\label{e:objective}
	\varphi(\boldsymbol{\alpha}):= \sup_{\mathbf{x} \in \calP_{q\times q}:\mathbf{x}\sim \boldsymbol{\alpha}}\{ - (d-1)H(\boldsymbol{\alpha}) +(d/2) [ \langle \log \mathbf{B}\rangle_{\mathbf{x}}  +H(\mathbf{x})] \} , 
\end{equation}
where $\mathbf{x}\sim \boldsymbol{\alpha}$ means that $\mathbf{x}$ is subject to the constraints
\begin{equation}
	\alpha_i  = \sum_{j} x_{ij}, \quad i \in [q].
\end{equation}

\begin{rem}\label{remark:equal_Bethe}
Recall from 
   \Cref{remark:model_transform} that matrix models on $d$-regular graphs are essentially equivalent to factor models. Starting with any matrix model $\mathbf{B}$, setting $\psi = \mathbf{B}$ and $\bar{\psi} \equiv 1$, it is immediate to check that 
   \eqref{e:objective} reduces to the Bethe rate function \eqref{e:Bethe_local}.  {In terms of the function $\varphi$, Proposition~\ref{proposition:rep_1} thus implies that} the annealed pressure {of a suitable class of spin systems} satisfies the constrained optimization problem:
 \begin{equation}\label{e:constrained_optimization}
	\varphi^{\mathrm{an}} = \sup_{\boldsymbol{\alpha}} \varphi(\boldsymbol{\alpha}) = \sup_{\boldsymbol{\alpha}} \sup_{\mathbf{x} \sim\, \boldsymbol{\alpha}} \Psi_{\mathrm{Bethe}}(\mathbf{x}). 
 \end{equation} 
 {Equation \eqref{e:constrained_optimization} shows that $\varphi$ can be interpreted as the Bethe rate function maximized over a given empirical vertex distribution $\boldsymbol{\alpha}$, where the variable $\mathbf{x}$ is again interpreted as the empirical edge distribution (cf.~\Cref{d:Beth_polytope}).} {Moreover, for each fixed $\boldsymbol{\alpha}$, the inner function is strictly concave in $\mathbf{x}$, so its stationarity conditions uniquely determine the inner maximizer. By the envelope theorem, the remaining Lagrange conditions are precisely the stationarity conditions for \(\varphi\). }
\end{rem}

In line with Remark~\ref{remark:equal_Bethe}, in \cite{BIS}, the expression \eqref{e:objective} is derived by estimating the asymptotics of the Boltzmann mass contributed by a fixed empirical vertex distribution.

\begin{lemma}
	[{\cite[Lemma 11]{BIS}}]\label{lemma:galanis} 
 Let $d \ge 2$. Define $g: [0,\infty)^q\setminus \{0\} \to {\mathcal P_q}$ whose $i$-th component is given by
 \begin{equation}
	g_i(\boldsymbol{u}) := \frac{u_i^{\frac{d}{d-1}}}{\sum_j u_j^{\frac{d}{d-1}} }.
 \end{equation} 
 For any ferromagnetic model, the stationary points of $\varphi$ and $\Phi$ can be put into one-to-one correspondence through 
	\begin{equation}
		\label{e:connection}
		\alpha_i = g_i(\boldsymbol{u}) 
	\end{equation}
	(regarding $\boldsymbol{u}$ as an equivalence class). Specifically, given $\boldsymbol{\alpha}$ that is a stationary point of $\varphi$, every vector in the half-ray $g^{-1}(\boldsymbol{\alpha})$ is a stationary point of $\Phi$. Conversely, given any $\boldsymbol{u} \in [0,\infty)^q\setminus\{\boldsymbol0\}$ that is a stationary point of $\Phi$, $g(\boldsymbol{u}) $ is a stationary point of $\varphi$. Moreover, $g$ restricted to the set of stationary points of $\Phi$ is a bijection, modulo the half-ray equivalence relation. 
	At the corresponding stationary points, where $\boldsymbol{\alpha} = g(\boldsymbol{u})$, the two functions satisfy 
	\begin{equation}
		\varphi(\boldsymbol{\alpha}) = \Phi(\boldsymbol{u}). 
	\end{equation}
\end{lemma}

\begin{proof}[Proof of \Cref{lemma:galanis}] 
The proof is adapted from \cite[Section 4.1]{Galanis}. We first derive the stationarity conditions of the two functions $\varphi$ and $\Phi$  through differentiation and then show that they are equivalent (after choosing a suitable representative for each stationary half-ray of $\Phi$). {For notational simplicity, throughout the proof, we abbreviate $p := d/(d-1)$.} By Remark~\ref{remark:equal_Bethe}, optimizing $\varphi$ reduces to optimizing $\Psi_{\mathrm{Bethe}}$.  We invoke the Lagrange multiplier method to solve the constrained problem \eqref{e:constrained_optimization} and identify the stationary condition for the function $\varphi$  defined in \eqref{e:objective}. 
Define the Lagrangian of $\Psi_{\mathrm{Bethe}}$ to be 
\begin{align}
\mathcal L(\mathbf{x},\boldsymbol{\alpha}, \boldsymbol{\lambda},\mu) &:= (d-1)\sum_i\alpha_i\log\alpha_i +\frac d2\sum_{i,j}x_{ij} \bigl(\log B_{ij}-\log x_{ij}\bigr) -\sum_i\lambda_i \left(\sum_jx_{ij}-\alpha_i\right) {-\mu\left(\sum_{i,j}x_{ij}-1\right)},
\end{align}	
where we parametrize \(\mathbf{x}\) through the entries \(x_{ij}\), \(1\leq i\leq j\leq q\), and set \(x_{ji}=x_{ij}\) for \(i<j\). 
In particular, there is no need for additional multipliers accounting for the condition that $\mathbf{x}$ is symmetric. 
Differentiating with respect to the off-diagonal variables \(x_{ij}\), \(i<j\), and with respect to \(\alpha_i\), \(\lambda_i\), and \(\mu\), yields the Lagrange conditions
\begin{align}
	\frac{\partial\calL}{\partial x_{ij}}(\mathbf{x},\boldsymbol{\alpha}, \boldsymbol{\lambda},\mu)  &=  d(\log B_{ij} - \log x_{ij}  -1) - \lambda_i -\lambda_j - 2 \mu = 0. \label{e:Lagrange_x} \\
    \frac{\partial\calL}{\partial \alpha_i}(\mathbf{x},\boldsymbol{\alpha}, \boldsymbol{\lambda},\mu)  &= (d-1)(\log \alpha_i +1)+\lambda_i = 0, \label{e:Lagrange_al} \\
     \frac{\partial\calL}{\partial \lambda_i}(\mathbf{x},\boldsymbol{\alpha}, \boldsymbol{\lambda},\mu)  &= -\left(\sum_jx_{ij}-\alpha_i \right)=0, \label{e:Lagrange_la} \\
     \frac{\partial\calL}{\partial \mu}(\mathbf{x},\boldsymbol{\alpha}, \boldsymbol{\lambda},\mu)  &= - \left(\sum_{i,j}x_{ij}-1\right) = 0. \label{e:Lagrange_mu}
\end{align}
For $i=j$, $\frac{\partial\calL}{\partial x_{ij}}(\mathbf{x},\boldsymbol{\alpha}, \boldsymbol{\lambda},\mu)  =  \frac{d}{2}(\log B_{ij} - \log x_{ij}  -1) - \lambda_i  -  \mu$, and setting this derivative to zero and multiplying by $2$ gives back condition \eqref{e:Lagrange_x}, which is why we only work with that condition for all $1\leq i \leq j \leq q$ in the following. 
Let $ u_i = \exp(- (\lambda_i{+\mu}) /d  -1/2)$ for $i \in [q]$. Then \eqref{e:Lagrange_al} and \eqref{e:Lagrange_x} become 
\[
\alpha_i = \exp\left(\frac{p\mu}{d}+\frac{p}{2}-1\right) u_i^p \quad \text{and }\quad x_{ij} = B_{ij} u_i u_j.
\]
Because $\boldsymbol{\alpha}$ is a probability vector, the last two identities reduce to 
\begin{equation}\label{e:Lagrange_condition}
	\alpha_i = \frac{u_i^p}{\sum_j u_j^p }\quad \text{and }\quad x_{ij} = B_{ij} u_i u_j  . 
\end{equation} 
Plugging \eqref{e:Lagrange_condition} into \eqref{e:Lagrange_la} and \eqref{e:Lagrange_mu}, respectively, yields
\begin{equation}\label{e:marginal_constraint}
	\frac{u_i^p}{\sum_j u_j^p } = u_i \sum_j B_{ij} u_j,
\end{equation}
and 
\begin{equation}\label{e:normalization}
	1 =  \sum_{i,j} B_{ij} u_i u_j.
\end{equation}

We conclude that a stationary point $\boldsymbol{\alpha}$ of $\varphi$ must be specified by some vector $\boldsymbol{u}$ through \eqref{e:Lagrange_condition} that satisfies the constraints \eqref{e:marginal_constraint} and \eqref{e:normalization}. 

On the other hand, differentiating the function $\Phi$ defined in \eqref{e:magic}   gives the stationarity conditions
\[
\frac{\partial \Phi}{\partial u_i}(\boldsymbol{u})  = d\frac{\sum_j B_{ij}u_j}{\sum_{k,l}B_{kl}u_ku_l} -  (d-1) \frac{p u_i^{{p-1}}}{\sum_j u_j^p } = 0.
\] 
Simplifying this equation for $i \in [q]$ gives
\begin{equation}\label{e:stationary_Phi}
	\frac{\sum_j B_{ij}u_j}{\sum_{k,l}B_{kl}u_ku_l} = \frac{ u_i^{{p-1} }}{\sum_j u_j^p }.
\end{equation}
Recall that $\Phi$ is invariant under rescaling $\mathbf{u}$ by positive constants, and \eqref{e:stationary_Phi} is preserved under such scaling as well. Moreover, if \(u_i=0\) for some \(i\), then strict positivity of \(\mathbf B\) implies that the one-sided derivative of \(\Phi\) in the direction of the $i$th unit vector is strictly positive; hence every stationary point of \(\Phi\) is strictly positive.
Since \(\mathbf B\) is strictly positive, every nonzero stationary point $\boldsymbol{u}$ can be rescaled so that \(\sum_{k,l}B_{kl}u_ku_l=1.\) For this representative, multiplying \eqref{e:stationary_Phi} by \(u_i\) yields 
\begin{equation}\label{e:recursion_2}
	u_i \sum_j B_{ij} u_j = \frac{u_i^p }{\sum_j u_j^p }.
\end{equation}
Thus, after fixing the scaling in this way, the stationarity condition for \(\Phi\) coincides with condition \eqref{e:marginal_constraint} obtained from the constrained optimization of \(\Psi_{\mathrm{Bethe}}\).

 Next, we show that the map \eqref{e:connection} gives a bijection between the stationary points of $\Phi$ and $\varphi$. Suppose $\boldsymbol{\alpha}$ is a stationary point of $\varphi$, thus corresponding to a stationary point $(\boldsymbol{x}, \boldsymbol{\alpha},\boldsymbol{\lambda},\mu)$ of $\calL$, and let $\boldsymbol{u}$ be defined by  $ u_i = \exp(- (\lambda_i{+\mu}) /d  -1/2)$. 
 Then $\boldsymbol{u}$ satisfies \eqref{e:Lagrange_condition} and \eqref{e:normalization}, and therefore  \eqref{e:recursion_2}.
 Thus, the vector $\boldsymbol{u}$ defined in this way is stationary for $\Phi$. Since \eqref{e:stationary_Phi} is preserved under positive scaling, every positive multiple of \(\boldsymbol{u}\), and hence every vector in \(g^{-1}(\boldsymbol{\alpha})\), is stationary for \(\Phi\). 
 
 Conversely, assume $\boldsymbol{u}$ is a stationary point of $\Phi$ 
 with \(\sum_{k,l}B_{kl}u_ku_l=1\) and {let $\boldsymbol{\alpha} = g(\boldsymbol{u})$.} Since $\boldsymbol{u}$ satisfies  \eqref{e:recursion_2}, it also satisfies the identical condition \eqref{e:marginal_constraint} as well as \eqref{e:normalization}. Moreover, setting $x_{ij} = B_{ij}u_iu_j$ yields a pair $(\boldsymbol{\alpha},\boldsymbol{x})$ that satisfies all conditions \eqref{e:Lagrange_condition} - \eqref{e:normalization}. Since the preceding change of variables is reversible, suitable multipliers \(\boldsymbol{\lambda}\) and \(\mu\) can be recovered, and the full Lagrange system is satisfied. We have thus constructed a stationary point $\boldsymbol{\alpha}$ of $\varphi$.
 This establishes the correspondence between the stationary points.  Moreover, \(g(\boldsymbol{u})=g(\boldsymbol{v})\) if and only if \(\boldsymbol{v}=c\boldsymbol{u}\) for some \(c>0\). Hence, \(g\) induces a bijection modulo the half-ray equivalence relation.
 
 Finally, we show that $\varphi(\boldsymbol{\alpha}) = \Phi(\boldsymbol{u})$ at stationary points related by \eqref{e:connection}. Let $\boldsymbol{\alpha}$ and $\boldsymbol{u}$ be a pair of corresponding stationary points. Then  
\begin{align*}
	\Phi(\boldsymbol{u}) &= \Phi\left(\frac{\boldsymbol{u}}{\sqrt{\sum_{k,l}B_{kl}u_ku_l}}\right) = 
    - (d-1) \log \left( \sum_{i=1}^q u_i^{{p}}\right) \stackrel{\eqref{e:Lagrange_condition}}{=} 
    (d-1) \sum_{i=1}^q \alpha_i  \log \left(\frac{\alpha_i}{u_i^p}\right) \\
	&\stackrel{\eqref{e:Lagrange_la}}{=} -(d-1)H(\boldsymbol{\alpha})- d\sum_{i,j}x_{ij} \log u_i = -(d-1)H(\boldsymbol{\alpha})- \frac{d}{2}\sum_{i,j=1}^q x_{ij} \log (u_iu_j) \\
	&\stackrel{\eqref{e:Lagrange_condition}}{=} -(d-1)H(\boldsymbol{\alpha})- \frac{d}{2}\sum_{i,j=1}^q x_{ij} \log\left( \frac{x_{ij}}{B_{ij}}\right) \stackrel{\text{Remark }\ref{remark:equal_Bethe}}{=}\varphi(\boldsymbol{\alpha}),
\end{align*}
where the fourth equality uses the symmetry of \(\mathbf{x}\). This completes the proof.
\end{proof}

\begin{rem}
	{Let $\boldsymbol{u} = (u_i)_{i \in [q]} \in (0,\infty)^q$ be the auxiliary vector associated with a stationary point $\boldsymbol{\alpha}$ of $\varphi$ via \eqref{e:Lagrange_condition}.} By combining \eqref{e:Lagrange_condition} and \eqref{e:marginal_constraint}, we get the fixed-point equation 
\begin{equation}\label{e:recursion}
	 u_i =\frac{1}{z_d}\left( \sum_{j }B_{ij} u_j\right)^{d-1}.
\end{equation}
	The fixed-point equation \eqref{e:recursion} that appears as a stationary condition for $\varphi(\boldsymbol{\alpha})$ turns out to exactly be the Bethe recursion defined in \Cref{d:BP}. This curious connection can be interpreted as an infinite analogue of \cite[Theorem 2]{BP}. See also the explanation below \Cref{d:Beth_polytope}. 
\end{rem}

\begin{proof}
	[Proof of \Cref{t:rep_3}] 
    First, observe that both $\varphi$ and $\Phi$ attain their maxima. By Remark~\ref{remark:equal_Bethe}, maximizing $\varphi$ is equivalent to maximizing $\Psi_{\mathrm{Bethe}}$, which is defined on the compact set $\calP_{q \times q}$, and thus attains its maxima. On the other hand, since $\Phi$ is invariant under rescaling $\mathbf{u}$ by positive constants, 
    \[
    \sup_{\boldsymbol{u} \in [0,\infty)^q \backslash \{0\}} \Phi(\boldsymbol{u}) = \sup_{\boldsymbol{u} \in [0,\infty)^q: \|u\|_2 = 1} \Phi(\boldsymbol{u}) = \max_{\boldsymbol{u} \in [0,\infty)^q: \|u\|_2 = 1} \Phi(\boldsymbol{u}),
    \]
    since the Euclidean unit sphere is compact.
    
    We next show that the local maxima of $\Phi$ and $\varphi$ must be stationary points, which allows us to use \Cref{lemma:galanis}. By Remark~\ref{remark:equal_Bethe}, for $\varphi$, it suffices to show that the local maxima of $\Psi_{\mathrm{Bethe}}$ are interior points, which was already done in the proof of Proposition~\ref{proposition:rep_1} (by Remark~\ref{remark:model_transform}, positivity of $\mathbf B$ implies positivity of the corresponding $\psi, \bar \psi$). For $\Phi$, {observe that it is sufficient to restrict to vectors $\mathbf{u}$ with $\|\mathbf u\|_2 = 1$, say.} We also proceed by contradiction. Suppose $\boldsymbol{u}$ is a local maximum of $\Phi$ with $\|\mathbf u\|_2 = 1$ and has some coordinate $u_i = 0$. We first note that by positivity of $\mathbf{B}$ and since $\mathbf u \not= 0$, $\sum_{ij} B_{ij} u_i u_j > 0$ and $\sum_{j} B_{ij} u_j > 0$. 
Therefore,
\[
\frac{\partial \Phi}{\partial u_i}(\boldsymbol{u})  = d\frac{\sum_j B_{ij}u_j}{\sum_{k,l}B_{kl}u_ku_l} -  (d-1) \frac{p u_i^{\frac{1}{d-1}}}{\sum_j u_j^{\frac{d}{d-1}} } = d\frac{\sum_j B_{ij}u_j}{\sum_{k,l}B_{kl}u_ku_l} >0.
\]
Hence, for all sufficiently small \(t>0\), \(\Phi(\boldsymbol{u}+t\boldsymbol{e}_i)>\Phi(\boldsymbol{u})\). By scale invariance, the same holds after normalizing \(\boldsymbol{u}+t\boldsymbol{e}_i\) to have unit norm, contradicting the local maximality of \(\boldsymbol{u}\) on \(\{\boldsymbol{v}\geq0:\|\boldsymbol{v}\|_2=1\}\).

Therefore, by \Cref{remark:equal_Bethe} and \Cref{lemma:galanis}, 
	\[
	\varphi^{\text{an}} = \max_{\boldsymbol{\alpha}} \varphi(\boldsymbol{\alpha}) = \max_{\boldsymbol{u} \ge 0: \boldsymbol{u} \ne 0} \Phi(\boldsymbol{u}). 
	\]
    Finally, by positive-semidefiniteness of $\mathbf{B}$, there exists $\hat{\mathbf{B}}$ with $\mathbf{B} = \hat{\mathbf{B}}^T\hat{\mathbf{B}}.$ 
    Recall that $p=d/(d-1)$. Since 
	\[
	\e^{\Phi(\boldsymbol{u}) } = \frac{\langle \boldsymbol{u},\mathbf{B} \boldsymbol{u}\rangle ^{d/2} }{||\boldsymbol{u}||_p^{d}} = \frac{\langle \boldsymbol{u},\hat{\mathbf{B}}^T \hat{\mathbf{B}} \boldsymbol{u}\rangle ^{d/2} }{||\boldsymbol{u}||_p^{d}} = \left( \frac{||\hat{\mathbf{B}}\boldsymbol{u}||_2 }{||\boldsymbol{u}||_p}\right)^d,
	\]
	we have 
	\[
	{\e^{\varphi^{\mathrm{an}}/d}} = \max_{ \boldsymbol{u} \ge 0: \boldsymbol{u} \ne 0 } \e^{\Phi(\boldsymbol{u})/d} = \max_{ \boldsymbol{u} \ge 0: \boldsymbol{u} \ne 0} \frac{||\hat{\mathbf{B}}\boldsymbol{u}||_2 }{||\boldsymbol{u}||_p} = ||\hat{\mathbf{B}}||_{p\to 2},
	\]
    which proves \eqref{e:annealed_norm}.
\end{proof}

\subsubsection{Second moment}\label{subsubsection:regular_second moment}
 In this section, we prove the existence of $\Psi_2$ from \Cref{d:second_moment} {for spin models on random regular graphs} and provide two representations for it. The results harness the results from the previous section along with the observation that the second moment of the partition function can be viewed as the partition function of a \textit{pair model} on the same random graph (see for example \cite{BIS}). 
 
 {Throughout the section, we work with the interaction-matrix models of \Cref{d:ferro} and again assume that $\calX=[q]$.}
 To formulate the main observation about $Z_n^2$, we first introduce the relevant matrix operation.
For two matrices $\mathbf{A},\mathbf{C}\in\bbR^{[q]\times [q]}$, their \emph{tensor product} $\mathbf{A}\otimes\mathbf{C}\in\bbR^{[q]^2\times [q]^2}$ is the matrix defined by
\begin{equation}\label{eq:tensor_product_definition}
(\mathbf{A}\otimes\mathbf{C})_{(a,b),(c,d)}
    :=A_{a,c}C_{b,d},
    \qquad (a,b),(c,d)\in [q]^2.
\end{equation}
In particular, $\mathbf{B}\otimes\mathbf{B}$ assigns  the product of their two original interaction weights to a pair of spins.

{Let $\mathbf{B}$ denote the interaction matrix of a spin model on a finite graph $G=(V,E)$.}
Set $\pmb{\Omega}:=([q]^2)^V$ and associate with each pair $\bsigma^{(1)},\bsigma^{(2)}\in\Omega$ the paired configuration $\pmb{\sigma}=(\bsigma^{(1)},\bsigma^{(2)})\in\pmb{\Omega}$.
Then the second moment of the partition function satisfies
 \begin{align*}
 Z_{G,\mathbf{B}}^2  &= \sum_{\bsigma^{(1)},\bsigma^{(2)} \in \Omega} \prod_{\{ i,j\} \in E} B_{\sigma^{(1)}_i,\sigma^{(1)}_j} \prod_{\{ s,t\} \in E} B_{\sigma^{(2)}_s,\sigma^{(2)}_t} = \sum_{\sigma^{(1)},\sigma^{(2)} \in \Omega} \prod_{\{ i,j\} \in E} (\mathbf{B}\otimes \mathbf{B})_{(\sigma^{(1)}_i, \sigma^{(2)}_i),(\sigma^{(1)}_j,\sigma^{(2)}_j)}\\
 &= \sum_{\pmb{\sigma} \in \pmb{\Omega}} \prod_{\{ i,j\} \in E} (\mathbf{B}\otimes \mathbf{B})_{\pmb{\sigma}_i,\pmb{\sigma}_j},
 \end{align*}
The last line is exactly the partition function of a model with spin set $[q]^2$ and interaction matrix $\mathbf{B}\otimes\mathbf{B}$ on the same graph $G$.
 
 \begin{remark}
    [Product of ferromagnetic models is ferromagnetic] \label{rem:ferro_product}
     Consider a ferromagnetic model specified by $\mathbf{B}$ with alphabet $[q]$. We show that the model specified by the tensor product $\mathbf{B}\otimes \mathbf{B}$ with alphabet $[q]\times [q]$ is ferromagnetic. This amounts to showing that the tensor product of a positive semidefinite matrix with itself is still positive semidefinite. To see this, let $\mathbf{P}$ be a change-of-basis matrix such that $\mathbf{P} \mathbf{B}\mathbf{P}^{-1} = \mathbf{\Lambda}$, where $\mathbf{\Lambda}$ is a diagonal matrix whose entries are all nonnegative. Therefore, $(\mathbf{P} \otimes \mathbf{P} ) (\mathbf{B}\otimes \mathbf{B})(\mathbf{P} \otimes \mathbf{P} )^{-1} = \mathbf{\Lambda} \otimes \mathbf{\Lambda}$. Clearly, the tensor product of a diagonal matrix whose entries are nonnegative with itself also has nonnegative entries, which proves that $\mathbf{B}\otimes \mathbf{B}$ is positive semidefinite.
 \end{remark}
 
The partition-function interpretation of $\Psi_2$ enables an application of the first-moment representations in \Cref{proposition:rep_1} and \Cref{t:rep_3} to $\Psi_2$ as well. {Using \Cref{remark:model_transform} to go between factor and matrix models on regular graphs,} we first state the Bethe variational representation of $\Psi_2$: 

\begin{proposition}[Bethe variational representation of $\Psi_2$] \label{proposition:second_moment_Bethe}
{Fix $d \ge 2$ and let $\pmb{\psi} = (\psi,\bar{\psi})$  be a factor model as in Definition~\ref{d:factor}
    on a sequence $(G_n)_n$ of random $d$-regular graphs generated by the configuration model. Let $\mathbf{B}$ be the corresponding interaction matrix and $\pmb{\psi}^{\otimes 2}$ be the factor model associated with $\mathbf B \otimes \mathbf B$ (cf.~\Cref{remark:model_transform}).}
Let $\Psi^{(2)}_{\mathrm{Bethe}}$ be the Bethe rate function in \eqref{e:Bethe_local} of the factor model $\pmb{\psi}^{\otimes 2}$ on $(G_n)_n$. 
Then the limit $\Psi_2:=\lim_{n\to\infty} \frac{1}{n}\log \bbE[Z_n^2]$ exists and is given by  
\[
\Psi_2 = \sup_{\mathbf{x} \in \calP_{q^2 \times q^2}} \Psi_{\mathrm{Bethe}}^{(2)}(\mathbf{x}).
\]
  Moreover, the maximum is attained at an interior point of $\calP_{q^2\times q^2}$.
\end{proposition}
\begin{proof} This is a direct consequence of the preceding observation and \Cref{proposition:rep_1}.
\end{proof}

Next, recall that given a ferromagnetic interaction matrix $\mathbf{B}$,  $\hat{\mathbf{B}} $ denotes any matrix such that $\mathbf{B} = \hat{\mathbf{B}}^T\hat{\mathbf{B}}$. Using \Cref{rem:ferro_product}, we state the representation of $\Psi_2$ corresponding to \Cref{t:rep_3}: 

\begin{proposition}[Matrix-norm representation of $\Psi_2$] \label{p:secondmoment_BIS}
    Let $d \ge 2$. For any ferromagnetic model on a {sequence of} random $d$-regular graphs with interaction matrix $\mathbf{B}$, 
    $\Psi_2$ exists and is given by 
	\begin{equation}
		\Psi_2 = \max_{\boldsymbol{u} \in [0,\infty)^{q^2}\backslash \{0\}} \Phi_{{\mathbf{B}^{\otimes 2}}}(\boldsymbol{u}).
	\end{equation}
	Here, $\Phi_{{\mathbf{B}^{\otimes 2}}}$ denotes the function $\Phi$ defined in \eqref{e:magic} with interaction matrix $\mathbf{B}^{\otimes 2}$. 
	In particular, 
	\begin{equation} \label{e:nice_sec_mo}
		\Psi_2= {d\log || \hat{\mathbf{B}} \otimes \hat{\mathbf{B}}||_{\frac{d}{d-1}\to 2} = 2d\log || \hat{\mathbf{B}} ||_{\frac{d}{d-1}\to 2}}
	\end{equation}
\end{proposition}
\begin{proof}
{This again follows from the previous observation and \Cref{t:rep_3}, using \Cref{rem:ferro_product} and the fact that $\hat{\mathbf{B}} \otimes \hat{\mathbf{B}}$ satisfies 
\begin{align*}
    \left(\hat{\mathbf{B}} \otimes \hat{\mathbf{B}}\right)^T \left(\hat{\mathbf{B}} \otimes \hat{\mathbf{B}}\right) = \mathbf{B} \otimes \mathbf{B}
\end{align*}
 and $|| \hat{\mathbf{B}} \otimes \hat{\mathbf{B}}||_{\frac{d}{d-1} \to 2}  = ||\hat{\mathbf{B}} ||_{\frac{d}{d-1} \to 2}^2$ (e.g., via  Bennett's tensor-product result \cite[Proposition~10.4]{Bennett}).}
\end{proof}

\subsection{{Quenched-annealed agreement and second moments}}\label[subsection]{subsection:regular_compare}

{Having focused in the previous section on representations of the annealed pressure, we now turn to its relation with the quenched pressure (recall~\Cref{d:ppp}). A striking feature of the ferromagnetic Potts model on random $d$-regular graphs is that its quenched and annealed pressures coincide throughout the parameter regime.  In \Cref{subsubsection:quenched_equals_anneealed}, we first review existing results on quenched--annealed agreement. Then, in \Cref{subsubsection:second_moment}, we give a direct proof based on the second-moment criterion in \Cref{proposition:Frieze}. Combined with concentration of the quenched pressure, this relation yields a direct second-moment proof of quenched--annealed agreement for ferromagnetic models on random $d$-regular graphs.}


\subsubsection{Quenched pressure equals annealed}\label{subsubsection:quenched_equals_anneealed}

Several proofs of the agreement of the quenched and annealed pressures on regular graphs are available, based on different methods and applying at different levels of generality. In this section, we summarize their common conclusion for the Potts model in \Cref{p:q_is_a} and subsequently review approaches based on the Bethe prediction, small subgraph conditioning, and the random-cluster representation. 

\begin{proposition}[ {\cite{Dembo,Dembo2,Sun,can,BIS,Hungarian,Remco_rcm}} ]\label{p:q_is_a}
{Let $\beta,B\geq 0$, $d\geq 2$, and $q\geq 2$. For the $q$-state Potts model on a random $d$-regular graph, the thermodynamic limits of the annealed and quenched pressures exist. Denoting these limits by $\varphi^{\mathrm{an}-\mathrm{Potts}}(d,q,\beta,B)$ and $\varphi^{\mathrm{qu}-\mathrm{Potts}}(d,q,\beta,B)$, respectively, it holds true that}
	\begin{equation}\label{e:Potts}
		\varphi^{\mathrm{an}-\mathrm{Potts}} (d,q,\beta,B)=\varphi^{\mathrm{qu}-\mathrm{Potts}}(d,q,\beta,B).
	\end{equation}
\end{proposition}

	The Ising part $(q=2)$ is shown in \cite[Proposition 3.2]{can} (explained in more detail in  \Cref{s:ising models}) and also in \cite{Dembo,Dembo2,Sun}, as explained in the discussion below \Cref{theorem:bethe_variational}. The equality \eqref{e:Potts} is proved in three different lines of research. Dembo, Montanari and Sun~\cite{Dembo2,Sun} use Bethe interpolation, Galanis et al.~\cite{BIS} prove this for a general ferromagnetic model by showing the existence of stationary points with a negative definite Hessian, and Can and van der Hofstad~\cite[Corollary 1.15]{Remco_rcm} prove it again by exploiting the connection between the random-cluster partition function and that of the Potts model \cite[Lemma 1.1]{Remco_rcm}.

The two pressures are naturally approached in different ways. The quenched pressure describes the behavior of the partition function for a typical graph realization. For locally tree-like graph sequences (see \Cref{s:local_def}), local weak convergence provides a useful framework for studying its thermodynamic limit, while on specific random graph models, methods such as small subgraph conditioning can provide finer information about its relation to the annealed pressure. The annealed pressure, by contrast, involves averaging the partition function over the random graph and may therefore depend on features of the graph distribution that are not captured by the local limit. On random regular graphs, it admits the variational representations reviewed in the previous section. Although both pressures admit variational characterizations in the settings considered below, these characterizations generally take rather different forms, and their equality is not transparent from the formulas alone.

{Although our focus is on random $d$-regular graphs, some of the results reviewed below concerning the quenched pressure apply more generally to locally tree-like graph sequences. We state these results in the form needed here and indicate their wider scope where relevant. We also note a difference in terminology: in \cite{Dembo2,Sun}, the normalized logarithm of the partition function is referred to as the \emph{free energy density}, whereas throughout this paper we call this quantity the \emph{pressure}. We nevertheless retain the standard terms \emph{Bethe free-energy functional} and \emph{Bethe free-energy prediction} below.}

\paragraph{\bf Bethe free energy prediction.} 
{The first approach compares two different variational formulations. The Bethe prediction for the quenched pressure is expressed in terms of the Bethe free-energy functional evaluated at fixed points of the belief-propagation recursion, whereas the annealed pressure on random $d$-regular graphs admits the variational representation over empirical spin and edge distributions established in \Cref{proposition:rep_1}. Although these formulations involve apparently different optimization problems, Dembo, Montanari, Sly and Sun~\cite{Dembo2,Sun} establish a correspondence between them and show that their optimal values agree. This allows results comparing the quenched pressure with the Bethe prediction to be combined with the annealed variational formula. We follow the formulation of \cite{Sun} for the general factor model $(\psi,\bar{\psi})$ introduced in \Cref{d:factor}.}

\begin{definition}[{Belief propagation recursion} {\cite[Definition 1.3]{Sun}}]\label{d:BP} 
	The \textit{belief propagation} map $\mathsf{BP}: \calP_q\to \calP_q$ {associated with the factor model on} the infinite $d$-regular tree is defined as
\begin{equation}
  (\mathsf{BP} \boldsymbol{h}) (\sigma):= \frac{1}{z_{\boldsymbol{h}}} \bar{\psi}(\sigma) 
  \left( \sum_{\sigma'}\psi(\sigma,\sigma')  \boldsymbol{h} (\sigma')\right)^{d-1}, \qquad \boldsymbol{h} \in \calP_q, \sigma\in [q],
\end{equation}
	where $z_{\boldsymbol{h}}$ is the normalizing constant that makes the right-hand side a probability measure, i.e.,
	\begin{equation}
		z_{\boldsymbol{h}}= \sum_{\sigma\in [q]} \bar{\psi}(\sigma)	\left( \sum_{\sigma'}\psi(\sigma,\sigma')  \boldsymbol{h} (\sigma')\right)^{d-1}.
	\end{equation}
\end{definition}

{
A sequence of random $d$-regular graphs converges locally to the infinite $d$-regular tree. The belief-propagation recursion is motivated by a consistency relation for marginals of the factor model on the latter. After removing one edge incident to a vertex, the remaining component rooted at that vertex has $d-1$ branches. Conditional on the spin at the root, the contributions from these branches factorize, leading to the map $\mathsf{BP}$ above. Thus, if $\boldsymbol{h}$ describes the same cavity marginal on each branch, then consistency with the recursive structure of the tree requires
\( \mathsf{BP}\boldsymbol{h}=\boldsymbol{h}. \) }
{
The Bethe free-energy functional introduced below is parametrized by probability distributions $\boldsymbol{h}\in\calP_q$, interpreted as candidate cavity marginals.
}

\begin{definition}[Bethe free energy functional {\cite[Definition 1.5]{Sun}}] \label{d:Bethe_free}The \textit{Bethe free energy functional} of the factor model on the infinite $d$-regular tree is the map $\Phi_{\mathrm{Bethe}}:\calP_q\to \bbR$ defined by 
\begin{equation}
  \Phi_{\mathrm{Bethe}}(\boldsymbol{h})  :=  \log \left[ \sum_\sigma \bar{\psi}(\sigma )\left( \sum_{\sigma'}\psi(\sigma, \sigma')\boldsymbol{h}(\sigma') \right)^d\right] - \frac{d}{2}\log \left( \sum_{\sigma, \sigma'} \psi(\sigma , \sigma') \boldsymbol{h}(\sigma)\boldsymbol{h}(\sigma')\right) .
\end{equation}
\end{definition}
The \emph{Bethe prediction} for the quenched pressure is obtained by restricting to those $\boldsymbol{h}$ that satisfy the fixed-point equation:
\begin{equation}
	\Phi_{\mathrm{Bethe}}^* :=  \sup \{ \Phi_{\mathrm{Bethe}}(\boldsymbol{h}): \boldsymbol{h} \text{ is a fixed point of } \mathsf{BP} \}.
\end{equation} 
We say that the factor model on a sequence of random graphs $(G_n)_n$ \textit{satisfies the Bethe prediction} if the thermodynamic limit $\varphi^{\mathrm{qu}}$ exists and satisfies \cite[(7)]{Sun}
\begin{equation}
	 \varphi^{\text{qu}}=  \Phi_{\mathrm{Bethe}}^* . 
\end{equation}

In \Cref{proposition:rep_1}, we expressed the annealed pressure on random $d$-regular graphs as a variational problem over empirical spin and edge distributions. The following result identifies the value of this variational problem with the Bethe prediction:

\begin{theorem}[{ \cite[Proposition 1.7]{Sun}}]  \label{theorem:bethe_variational}
	 For factor models on random $d$-regular graphs, the corresponding Bethe variational problems have the same optimal value: 
	\[
	 \Phi_{\mathrm{Bethe}}^*  
     = \sup_{\mathbf{x} \in \calP_{q \times q}} \Psi_{\mathrm{Bethe}}(\mathbf{x}).
	\]
	Moreover, the local maxima (which are interior points) of $\Psi_{\mathrm{Bethe}}$ are in one-to-one correspondence with the fixed points of the belief-propagation map $\mathsf{BP}$. 
\end{theorem} 

{A more general version of this variational correspondence is given in \cite[Theorem~1.18]{Dembo2} for locally tree-like graphs. The correspondence itself is a property of the factor model on the limiting tree and does not identify the limiting quenched pressure of a particular graph sequence. However, in the random $d$-regular setting considered here, \Cref{proposition:rep_1} and \Cref{theorem:bethe_variational} imply that \begin{equation}\label{e:annealed_equals_Bethe} \varphi^{\mathrm{an}}=\Phi_{\mathrm{Bethe}}^*. \end{equation} It remains to compare the quenched pressure with $\Phi_{\mathrm{Bethe}}^*$.}

For the Ising model ($q=2$), Dembo, Montanari and Sun~\cite[Theorem~1.9]{Dembo2} verify the Bethe prediction for all $\beta,B\geq0$. Hence,
 by \Cref{theorem:bethe_variational}, we get 
\[
\varphi^{\mathrm{qu}-\mathrm{Ising}}(\beta,B)
=\varphi^{\mathrm{an}-\mathrm{Ising}}(\beta,B),
\quad \beta,B\ge 0.
\]
For integer $q\geq3$, Dembo, Montanari and Sun~\cite[Theorem~1.10]{Dembo2} only establish that the Bethe prediction is a lower bound for the quenched pressure:
\[
   \Phi^*_{\mathrm{Bethe}} \le \varphi^{\mathrm{qu}-\mathrm{Potts}}.
\]
However, in combination with \Cref{theorem:bethe_variational} and Jensen's inequality, for random regular graphs, this answers the agreement problem between the quenched and annealed Potts models:
\[
 \varphi^{\mathrm{an}-\mathrm{Potts}}
 =\Phi^*_{\mathrm{Bethe}}
 \le\varphi^{\mathrm{qu}-\mathrm{Potts}}
 \le\varphi^{\mathrm{an}-\mathrm{Potts}}.
\]

\paragraph{\bf Small subgraph conditioning.} A second approach, developed by Galanis et al.~\cite{BIS}, combines a second-moment analysis with the small subgraph conditioning technique. Unlike the preceding approach, it applies to general ferromagnetic matrix models on random $d$-regular graphs, subject to a nondegeneracy condition on a maximizing empirical spin distribution. 

Recall the function $\varphi$ defined in \eqref{e:objective}. If $Z_n(\boldsymbol{\alpha})$ denotes the contribution to the partition function from configurations with empirical spin distribution $\boldsymbol{\alpha}$, then $\varphi(\boldsymbol{\alpha})$ is the exponential growth rate of the expected contribution from configurations with empirical spin distribution $\boldsymbol{\alpha}$:
\[ \bbE[Z_n(\boldsymbol{\alpha})] = \e^{n\varphi(\boldsymbol{\alpha})+o(n)}. \]

Following \cite{BIS}, a maximizer $\boldsymbol{\alpha}$ of $\varphi$ is called a \emph{dominant phase}. Such a phase is called \emph{Hessian dominant} if the Hessian of $\varphi$, restricted to the tangent space of the simplex $\calP_q$, is negative definite at $\boldsymbol{\alpha}$. 

For a dominant phase $\boldsymbol{\alpha}$, Galanis et al.~\cite[Theorem~7]{BIS} show that the second moment of $Z_n(\boldsymbol{\alpha})$ has the same exponential growth rate as the square of its first moment. This agreement at exponential order alone does not imply that $Z_n(\boldsymbol{\alpha})$ is close to its expectation with high probability. In particular, fluctuations in the numbers of short cycles can cause the ratio \[ \frac{\bbE[Z_n(\boldsymbol{\alpha})^2]} {\bbE[Z_n(\boldsymbol{\alpha})]^2} \] to remain bounded away from one. Small subgraph conditioning accounts for these fluctuations by conditioning on the numbers of short cycles in the random graph. Under the Hessian-dominance assumption, it yields a high-probability lower bound on $Z_n(\boldsymbol{\alpha})$ that differs from its expectation only by a subexponential factor. After taking the normalized logarithm, this factor disappears, so the contribution from a dominant phase attains the annealed pressure with high probability. This gives the following result:

\begin{theorem}[{\cite[Theorem~8]{BIS}}] \label{theorem:BIS_agreement} Let $d\geq3$, and consider a ferromagnetic matrix model on a random $d$-regular graph. If the function $\varphi$ defined in \eqref{e:objective} admits a Hessian dominant maximizer, then the limiting quenched and annealed pressures coincide: \[ \varphi^{\mathrm{qu}}=\varphi^{\mathrm{an}}. \] \end{theorem}

The argument does not apply directly when $d=2$, as the verification of the small subgraph conditioning assumptions involves a matrix that becomes singular in this case. In particular, the proof does not establish condition~(A4) of \cite[Theorem~37]{BIS} for $d=2$. 

For the Potts model, the phase-diagram analysis in \cite[Theorem~5]{BIS} identifies the dominant phases and determines when they are Hessian dominant. Combined with \Cref{theorem:BIS_agreement}, this yields quenched--annealed agreement for all $\beta,B\geq0$, except possibly for a single value of $\beta$, where the required Hessian nondegeneracy may fail. 

\begin{corollary}\label{theorem:Potts_BIS} For any $d\geq3$ and $q\geq3$, the quenched and annealed pressures of the $q$-state Potts model on a random $d$-regular graph coincide for all $\beta,B\geq0$, except possibly when $\beta=\beta_u$, the uniqueness critical point. \end{corollary}

\paragraph{\bf Random cluster representation.} 
A third approach exploits the correspondence between the Potts and random-cluster partition functions; see \cite[Lemma~1.1]{Remco_rcm}. Bencs, Borb\'{e}nyi and Csikv\'{a}ri \cite{Hungarian} show that, on locally tree-like $d$-regular graph sequences, the random-cluster partition function at zero external field can be approximated up to a subexponential factor by the partition function of a suitable ferromagnetic two-spin model; see \cite[Theorem~2.5]{Hungarian}. Consequently, the two partition functions have the same limiting normalized logarithm. Their annealed and quenched results for the associated two-spin model \cite[Theorems~3.16 and~3.31]{Hungarian} then identify this limit with the annealed pressure on random $d$-regular graphs, thereby establishing quenched--annealed and, in fact, to noninteger $q>2$. Building on this rank-two approximation, Can and van der Hofstad~\cite[Corollary~1.15]{Remco_rcm} subsequently extended the result to all $\beta,B\geq0$

\subsubsection{{Second moment proof of the quenched-annealed agreement on the configuration model}} \label{subsubsection:second_moment}

We now present a direct second-moment proof of quenched--annealed agreement for ferromagnetic matrix models on the random $d$-regular configuration model, based on the matrix-norm representation of \cite{Galanis} and Frieze's trick \Cref{proposition:Frieze}. Let $(\mathrm{CM}_n(d))_{n\geq1}$ be a sequence of random $d$-regular configuration models defined on the same probability space. For an interaction matrix $\mathbf{B}$, we denote the corresponding random quenched, quenched, and annealed pressures by $\varphi_{\mathrm{CM}_n}^{\mathrm{r\text{-}qu}}(\mathbf{B})$, $\varphi_{\mathrm{CM}_n}^{\mathrm{qu}}(\mathbf{B})$, and $\varphi_{\mathrm{CM}_n}^{\mathrm{an}}(\mathbf{B})$, respectively, as in \Cref{d:ppp}. Their limiting quenched and annealed pressures are denoted by $\varphi_{\mathrm{CM}}^{\mathrm{qu}}(\mathbf{B})$ and $\varphi_{\mathrm{CM}}^{\mathrm{an}}(\mathbf{B})$, respectively. The following theorem gives the precise version of the first part of the informal statement in \Cref{thm_main_inf}. 
The corresponding result for the uniform simple $d$-regular graph is derived in the following subsection.

\begin{theorem}[Second moment and quenched--annealed agreement for the configuration model]\label{t:main1} 
Let $d\geq2$, and let $\mathbf{B}$ be a strictly positive ferromagnetic interaction matrix. For the random $d$-regular configuration model, the limiting annealed pressure and normalized logarithmic second moment exist and satisfy 
\begin{equation}\label{e:match} 
\Psi_{2,\mathrm{CM}} = 2\varphi_{\mathrm{CM}}^{\mathrm{an}}(\mathbf{B}). 
\end{equation} 
Moreover, the random quenched pressure converges almost surely to the annealed pressure: 
\begin{equation}\label{e:random_pressure_CM} \varphi_{\mathrm{CM}_n}^{\mathrm{r\text{-}qu}}(\mathbf{B}) \xrightarrow{\mathrm{a.s.}} \varphi_{\mathrm{CM}}^{\mathrm{an}}(\mathbf{B}). 
\end{equation}
Consequently, the limiting quenched pressure exists and satisfies 
\begin{equation}\label{QA} 
\varphi_{\mathrm{CM}}^{\mathrm{qu}}(\mathbf{B}) = \varphi_{\mathrm{CM}}^{\mathrm{an}}(\mathbf{B}). 
\end{equation}
In particular, these conclusions hold for the ferromagnetic Ising and Potts models. 
\end{theorem} 

Before turning to the proof of \Cref{t:main1}, a few words on its relation to \cite{BIS} are in order. 
The scope of \Cref{t:main1} is complementary to that of \cite[Theorem~8]{BIS}, which allows ergodic interaction matrices with zero entries, but assumes $d\ge 3$ and the existence of a dominant phase with a negative-definite Hessian.
Our argument applies for $d\ge 2$ without this phase or Hessian assumption, but requires the matrix entries to be strictly positive.

In terms of proof approach, recall from the previous section that the partition function may be decomposed according to the empirical spin distribution as 
\[ Z_n=\sum_{\boldsymbol{\alpha}}Z_n(\boldsymbol{\alpha}), \] 
where $Z_n(\boldsymbol{\alpha})$ is the contribution from configurations with empirical spin distribution $\boldsymbol{\alpha}$. In~\cite[Theorem~7]{BIS}, Galanis et al.\ consider $Z_n(\boldsymbol{\alpha})$ for a dominant empirical spin distribution $\boldsymbol{\alpha}_{\mathrm{max}}$ and show that 
\[ \bbE[Z_n(\boldsymbol{\alpha}_{\mathrm{max}})^2] = \e^{o(n)}\bbE[Z_n(\boldsymbol{\alpha}_{\mathrm{max}})]^2 = \e^{o(n)}\bbE[Z_n]^2. \] 
This result does not directly give \eqref{e:match}, since it concerns the partition function restricted to a prescribed empirical spin distribution. Nevertheless, the proof of \cite[Theorem~7]{BIS} contains the matrix-norm calculation from which \eqref{e:match} follows; see \cite[proof of Theorem~7, equations~(22)--(25)]{BIS}.
We isolate this part and combine it with Frieze's trick as stated in \Cref{proposition:Frieze} to obtain the quenched--annealed agreement via a direct route from the second-moment criterion, without restricting the empirical spin distribution or applying small subgraph conditioning.

\begin{proof}[Proof of \Cref{t:main1}]
  The existence of the limiting annealed pressure and normalized logarithmic second moment was established in \Cref{subsection:regular_moment}. 
	From the matrix norm representations \eqref{e:annealed_norm} and \eqref{e:nice_sec_mo} of $\varphi^{\text{an}}$ and $\Psi_2$, we get 
	\[
	\Psi_2 =d\log || \hat{\mathbf{B}}_2||_{p\to 2} = d\log  (  ||\hat{\mathbf{B}}||_{p\to 2})^2 = 2 \varphi^{\text{an}}.
	\]
	This proves \eqref{e:match}.
It remains to verify the concentration hypothesis \eqref{e:concentration} of \Cref{proposition:Frieze}. We proceed as in \Cref{ex:Potts_con}. Let $m$ and $M$ denote the minimum and maximum entries of $\mathbf{B}$, respectively. If two pairings differ by a single switching, then the Boltzmann weight of every spin configuration changes by a factor in  \( \left[(m/M)^2,(M/m)^2\right]. \) Consequently, for the corresponding graphs $G$ and $G'$, 
\[ |\log Z_G-\log Z_{G'}| \leq 2\log(M/m). \] 
The switching coupling and Azuma--Hoeffding argument from \Cref{ex:Potts_con} therefore apply verbatim and yield \eqref{e:concentration}. Notice that this step uses the strict positivity of $\mathbf{B}$, but not its positive semidefiniteness. 
	Combining \eqref{e:match} and \eqref{e:concentration} with \Cref{proposition:Frieze} proves \eqref{e:random_pressure_CM} and \eqref{QA}.
	The Ising and Potts models are particular ferromagnetic models.
\end{proof}

We conclude with two observations concerning the second-moment identity \eqref{e:match}. The first relates it to the restricted second-moment calculation of \cite{BIS}, while the second describes a possible alternative derivation through the Bethe formulation.

\begin{rem} 
Let $\boldsymbol{\alpha}_{\mathrm{max}}$ be a dominant empirical spin distribution. Combining \eqref{e:match} with \cite[Theorem~7]{BIS} gives 
\[
\bbE[Z_n(\boldsymbol{\alpha}_{\mathrm{max}})^2] = \e^{o(n)}\bbE[Z_n]^2 = \e^{o(n)}\bbE[Z_n^2]. 
\] 
Thus, pairs of configurations for which both coordinate configurations have empirical spin distribution $\boldsymbol{\alpha}_{\mathrm{max}}$ contribute to the second moment at its full exponential rate. This clarifies how the restricted second-moment calculation in \cite{BIS} is related to the unrestricted identity \eqref{e:match}. 
\end{rem} 

\begin{rem}
The proof of \eqref{e:match} above uses the matrix-norm representations of the first and second moments in Frieze's trick. One could alternatively try to derive \eqref{e:match} through the Bethe formulation of the paired-spin model, whose partition function is $Z_n^2$. Indeed, \Cref{proposition:rep_1}, \Cref{proposition:second_moment_Bethe} and \Cref{theorem:bethe_variational} give 
\[
\Psi_2 = \sup_{\mathbf{x}\in\calP_{q^2\times q^2}} \Psi_{\mathrm{Bethe}}^{(2)}(\mathbf{x}) = \Phi_{\mathrm{Bethe}}^{*,(2)}, 
\] 
where $\Phi_{\mathrm{Bethe}}^{*,(2)}$ denotes the Bethe prediction for the paired-spin model. If this model satisfies the Bethe prediction, then 
\[ \Psi_2 = \varphi^{\mathrm{qu},(2)} = \lim_{n\to\infty}\frac{1}{n}\bbE[\log Z_n^2] = 2\varphi^{\mathrm{qu}}. \] 
Consequently, once quenched--annealed agreement for the original model is known, this gives 
\[ \Psi_2=2\varphi^{\mathrm{an}}. \] 
This route, however, requires verifying the Bethe prediction for the paired-spin model, which does not follow immediately from the corresponding statement for the original model. The matrix-norm argument avoids this additional step.
\end{rem}

\subsubsection{Transfer {of the quenched-annealed agreement to uniform regular simple graphs}}

We now transfer the configuration-model conclusions to uniform simple regular graphs.  For a sequence $(\mathrm{UG}_n(d))_{n \geq 1}$ of uniform simple $d$-regular graphs that are defined on the same probability space. For an interaction matrix $\mathbf{B}$, we denote the corresponding random quenched, quenched, and annealed pressures by $\varphi_{\mathrm{UG}_n}^{\mathrm{r\text{-}qu}}(\mathbf{B})$, $\varphi_{\mathrm{UG}_n}^{\mathrm{qu}}(\mathbf{B})$, and $\varphi_{\mathrm{UG}_n}^{\mathrm{an}}(\mathbf{B})$, respectively, as in \Cref{d:ppp}. Their limiting quenched and annealed pressures are denoted by $\varphi_{\mathrm{UG}}^{\mathrm{qu}}(\mathbf{B})$ and $\varphi_{\mathrm{UG}}^{\mathrm{an}}(\mathbf{B})$, respectively. The following theorem shows that these latter notions are well-defined: 

\begin{theorem}[{Quenched $=$ annealed for the uniform simple regular graph}]\label{t:main2}
	Let $d\ge 2$. Consider a spin model with strictly positive, ferromagnetic interaction matrix $\mathbf{B}$, and let $(\mathrm{UG}_n(d))_{n \geq 1}$ be a sequence of uniform simple $d$-regular graphs. Then $\Psi_{2,\mathrm{UG}} = 2\varphi_{\mathrm{UG}}^{\mathrm{an}}(\mathbf{B}),$ $ \varphi_{\mathrm{UG}}^{\mathrm{an}}(\mathbf{B})  = \varphi_{\mathrm{CM}}^{\mathrm{an}}(\mathbf{B})$ and
    \begin{equation}\label{QA2}
		\varphi_{\mathrm{UG}}^{\mathrm{an}}(\mathbf{B})
		=\varphi_{\mathrm{UG}}^{\mathrm{qu}}(\mathbf{B}) \qquad \text{and} \qquad \varphi_{\mathrm{UG}_n}^{\mathrm{r-qu}}(\mathbf{B}) \stackrel{a.s.}{\longrightarrow}
		\varphi_{\mathrm{UG}}^{\mathrm{an}}(\mathbf{B}).
	\end{equation}
	In particular, these conclusions hold for the ferromagnetic Ising and Potts models.
\end{theorem}

The main step in establishing \Cref{t:main2} is the following lemma, which we prove first:

\begin{lemma}\label{lemma:config2simple}
	Let $d \ge 2.$  
    Consider a spin model with strictly positive, ferromagnetic interaction matrix $\mathbf{B}$, and let $(\mathrm{UG}_n(d))_{n \geq 1}$ be a sequence of uniform simple $d$-regular graphs. Then the limiting annealed pressure exists and satisfies
    \begin{align}
       \varphi_{\mathrm{UG}}^{\mathrm{an}}(\mathbf{B})  = \varphi_{\mathrm{CM}}^{\mathrm{an}}(\mathbf{B}). 
    \end{align}
\end{lemma}
\begin{proof}
   Fix a spin model with strictly positive ferromagnetic interaction matrix $\mathbf{B}$. For notational simplicity, we omit the $\mathbf{B}$-dependence in the notation (e.g. $\varphi_{\mathrm{UG}}^{\mathrm{an}}(\mathbf{B})$ becomes $\varphi_{\mathrm{UG}}^{\mathrm{an}}$) and write $\bbE[X; A]\equiv \bbE[X\mathds{1}_A]$. All expectations and probabilities are taken with respect to the configuration models $(\mathrm{CM}_n(d))_{n \geq 1}$.  Let
    $$A_n = \{ \mathrm{CM}_n \text{ is simple}\}, \qquad \text{such that} \qquad \bbP(A_n)\longrightarrow c_d \in (0,1)$$
    (see \cite[Proposition 7.12]{vdH_2016}). Conditionally on $A_n$, the configuration model has the law of $\mathrm{UG}_n(d)$, and 
 \[ \varphi_{\mathrm{UG}_n}^{\mathrm{an}} = \frac{1}{n} \log\frac{\bbE[Z_n;A_n]}{\bbP(A_n)}. \] 
  We first show the upper bound 
    $\limsup_{n\to\infty} \varphi_{\mathrm{UG}_n}^{\mathrm{an}}  \le \varphi_{\mathrm{CM}}^{\mathrm{an}}$. By definition of the law of $(\mathrm{UG}_n(d))_{n \geq 1}$,
	\[
	\varphi_{\mathrm{UG}_n}^{\mathrm{an}} = \frac{1}{n} \log \left(\frac{\bbE[Z_n; A_n]}{\bbP(A_n)}\right) = \frac{1}{n}\log \bbE[Z_n ;A_n] + o(1) \le \varphi_{\mathrm{CM}_n}^{\mathrm{an}} +o(1).
	\]
	Taking $n\to\infty$ on both sides shows that $\limsup_{n\to\infty} \varphi_{\mathrm{UG}_n}^{\mathrm{an}}  \le \varphi_{\mathrm{CM}}^{\mathrm{an}}$. 
    
   For the matching lower bound, fix $\varepsilon>0$ and let
\( E_n := \left\{ \log Z_n/n \geq \varphi_{\mathrm{CM}}^{\mathrm{an}}-\varepsilon \right\}. \)
 By \eqref{e:random_pressure_CM}, $\bbP(E_n^c)\to0$. Since $\bbP(A_n)\to c_d\in(0,1)$, 
 \[ \bbP(A_n\cap E_n) \geq \bbP(A_n)-\bbP(E_n^c) = c_d+o(1). \]
 Consequently, 
\begin{align*}
\bbE[Z_n;A_n] \geq \bbE[Z_n;A_n\cap E_n] \geq \e^{n(\varphi_{\mathrm{CM}}^{\mathrm{an}}-\varepsilon)} \bbP(A_n\cap E_n) = \e^{n(\varphi_{\mathrm{CM}}^{\mathrm{an}}-\varepsilon)+o(n)}.
\end{align*}
It follows that 
\( \liminf_{n\to\infty} \varphi_{\mathrm{UG}_n}^{\mathrm{an}} \geq \varphi_{\mathrm{CM}}^{\mathrm{an}}-\varepsilon. \)
Letting $\ve \downarrow0$ proves the lower bound.
\end{proof}

\begin{proof}[Proof of \Cref{t:main2}]
Fix a strictly positive ferromagnetic interaction matrix $\mathbf{B}$ and suppress it from the notation. By \Cref{lemma:config2simple}, \(\varphi_{\mathrm{UG}}^{\mathrm{an}} = \varphi_{\mathrm{CM}}^{\mathrm{an}}. \) 

Consider next the paired-spin model with interaction matrix $\mathbf{B}^{\otimes 2}$. This model is again strictly positive and ferromagnetic, and its partition function is $Z_n^2$. Applying \Cref{lemma:config2simple} to the paired-spin model and using \Cref{t:main1}, we obtain 
\[ \Psi_{2,\mathrm{UG}} = \Psi_{2,\mathrm{CM}} = 2\varphi_{\mathrm{CM}}^{\mathrm{an}} = 2\varphi_{\mathrm{UG}}^{\mathrm{an}}. \]

To apply \Cref{proposition:Frieze} once more, it remains to verify concentration. Let \( A_n:=\{\mathrm{CM}_n(d)\text{ is simple}\}\) as in the previous proof, where $\bbP(A_n)\to c_d\in(0,1)$. Hence, for every $\varepsilon>0$, the configuration-model concentration estimate \eqref{e:concentration} gives 
\begin{align*} 
\bbP\left( \left| \varphi_{\mathrm{UG}_n}^{\mathrm{r\text{-}qu}} - \varphi_{\mathrm{CM}_n}^{\mathrm{qu}} \right|>\varepsilon \right) = \bbP\left( \left| \varphi_{\mathrm{CM}_n}^{\mathrm{r\text{-}qu}} - \varphi_{\mathrm{CM}_n}^{\mathrm{qu}} \right|>\varepsilon \,\middle|\,A_n \right) \leq \frac{ \bbP\left( \left| \varphi_{\mathrm{CM}_n}^{\mathrm{r\text{-}qu}} - \varphi_{\mathrm{CM}_n}^{\mathrm{qu}} \right|>\varepsilon \right) }{ \bbP(A_n) } \leq \e^{-c_\varepsilon n},
\end{align*} 
after adjusting $c_\varepsilon>0$. Since the random quenched pressures are uniformly bounded, taking expectations in the preceding concentration estimate also gives
\[ \left| \varphi_{\mathrm{UG}_n}^{\mathrm{qu}} - \varphi_{\mathrm{CM}_n}^{\mathrm{qu}} \right| \longrightarrow0. \] 
Consequently, for every $\varepsilon>0$ and all sufficiently large $n$, 
\begin{align*} 
\bbP\left( \left| \varphi_{\mathrm{UG}_n}^{\mathrm{r\text{-}qu}} - \varphi_{\mathrm{UG}_n}^{\mathrm{qu}} \right|>2\varepsilon \right) \leq \bbP\left( \left| \varphi_{\mathrm{UG}_n}^{\mathrm{r\text{-}qu}} - \varphi_{\mathrm{CM}_n}^{\mathrm{qu}} \right|>\varepsilon \right) \leq \e^{-c_\varepsilon n}. 
\end{align*}
Thus, the required exponential concentration estimate holds for the uniform simple model. 
The second-moment identity above and \Cref{proposition:Frieze} now give 
\[ \varphi_{\mathrm{UG}_n}^{\mathrm{r\text{-}qu}} \xrightarrow{\mathrm{a.s.}} \varphi_{\mathrm{UG}}^{\mathrm{an}} \qquad\text{and}\qquad \varphi_{\mathrm{UG}}^{\mathrm{qu}} = \varphi_{\mathrm{UG}}^{\mathrm{an}}.
\] 
\end{proof}


\noindent
{\bf Acknowledgments.}
The work of RvdH and NM was supported in part by the Netherlands Organisation for Scientific Research (NWO) through the Gravitation {\small{\sf NETWORKS}} grant no.\ 024.002.003. The work of WMR is supported by the NWO VIDI grant VI.Vidi.213.112.

{\small
\bibliographystyle{alpha-sortname}  
\bibliography{bibliography.bib}}

\appendix
\section{The annealed Ising pressure on random regular graphs}
\label{s:ising models}
This appendix complements the general representations in
\Cref{subsection:regular_moment} by presenting two further representations specific
to the Ising model ($q=2$). {Although the Ising model can be viewed as the two-state Potts model, it is traditionally formulated with spin values in $\{-1,1\}$ instead of $\{1,2\}$, and we adopt this convention throughout the appendix.} The first representation, due to Can~\cite{can}, is based on the moment
generating function of the number of disagreeing edges and parametrizes the
pressure by the empirical proportion of positive spins. The second, due to
Borinsky, Ren and Wiesmann~\cite{BorinskyRenWiesmann}, uses the
high-temperature expansion of van der Waerden \cite{vanDerWaerden1941} and a weighted enumeration of edge-bicolored
regular graphs. 

\paragraph{\bf Ising variational formula via moment generating functions.} We first present Can's variational formula \cite{can} for the annealed pressure of the Ising model on random $d$-regular graphs for all $d\geq2$. 
Here, the optimization parameter $t$ represents 
the empirical proportion of positive spins.

\begin{theorem}[Theorem 1.1 of \cite{can}]\label{t:can_rep}
 On a random $d$-regular graph with $d \geq 2$, for all $\beta \geq 0$ and $B \in \mathbb{R}$, the annealed Ising pressure $\varphi_n^{\mathrm{an}-\mathrm{Ising}}$ converges, and its limit $\varphi^{\mathrm{an}-\mathrm{Ising}}(\beta, B)$ is given by
\begin{equation}\label{eq:annealed_press}
	\varphi^{\mathrm{an}-\mathrm{Ising}}(\beta, B) =\frac{\beta d}{2}-B+\max _{0 \leq t \leq 1}[H(t)+2 B t+d F_\beta(t)],
\end{equation}
where
\begin{equation} \label{eq:F}
	F_\beta(t)=\int_0^{u(t)} \log f_\beta(s) d s,
\end{equation}
with $u(t)=\min \{t, 1-t\}$ and
\begin{equation}\label{e:f}
	f_\beta(s)= \frac{\e^{-2 \beta}(1-2 s)+\sqrt{1+\left(\e^{-4 \beta}-1\right)(1-2 s)^2}}{2(1-s)}.
\end{equation}
\end{theorem}

Recall the representation \eqref{eq:Bolt_weight} of the contribution of a single configuration $\sigma$ to the partition function. Below in \eqref{eq_Z_posmono}, we will see that for the ferromagnetic Ising model with spins in $\{-1,1\}$, this simplifies to an expression in terms of the total number of positive spins and edges between disagreeing spins. As the following proof sketch shows, Can's formula is based on the observation that on random $d$-regular graphs, the \emph{law} of the number of edges between disagreeing spins itself only depends on the total number of positive spins. This enables the expression \eqref{eq:annealed_press} of the annealed pressure to be written solely in terms of an optimization problem over the proportion of positive spins. The main contribution of Can is to obtain an explicit asymptotic formula for the moment generating function of the number of edges between disagreeing spins in terms of the proportion of positive spins.

\begin{proof}[Sketch of the proof of Theorem \ref{t:can_rep}]  For a spin configuration $\bsigma\in \{\pm 1\}^V$, let $\sigma_+$ and $\sigma_-$ denote
the sets of vertices carrying spins $+1$ and $-1$, respectively, and let
$e(\sigma_+,\sigma_-)$ be the number of edges between these two sets, which we call \emph{disagreeing edges}. 
Starting from the partition function, we rewrite 
	\begin{align}\label{eq_Z_posmono}
	Z_n = \sum_{\bsigma\in \Omega} \e^{ \beta \sum_{\{ i, j \} \in E_n} \sigma_i \sigma_j + B\sum_{i \in V } \sigma_i  }=\sum_{\bsigma\in \Omega} \e^{\beta |E_n| -B |V_n| } \e^{-2\beta e(\sigma_+, \sigma_- ) + 2|\sigma_+| B} .
	\end{align}	
We next take a closer look at the distribution of $e(\sigma_+, \sigma_-)$. For this, let $X(k,m)$ denote the number of edges between vertices in $[k]$ and $[m]\setminus [k]$ in the $1$-regular configuration model $\mathrm{CM}_m(1)$. 
By the definition of the configuration model as a successive uniform pairing of half-edges, $e(\sigma_+,\sigma_-)$ is distributed as  $X(d|\sigma_+|, dn)$ (see \cite[(2.2)]{can}). In particular, its distribution
depends on $\sigma$ only through
$j=|\sigma_+|$. 
	Taking expectations and grouping configurations by their number of positive spins gives
	\begin{equation}\label{eq:e_to_X}
		\bbE[Z_n] = \e^{\beta nd/2 - Bn} \sum_{j=0}^{n} { n\choose j}\e^{2jB} \bbE[\e^{-2\beta X(dj,dn)}].  
	\end{equation}
For convenience, write $g_\beta(k,m):=\bbE[\e^{-2\beta X(k,m)}]$, so that \eqref{eq:e_to_X} becomes
\begin{equation}\label{e:Z}
	\bbE[Z_n]
    =
    \e^{\beta dn/2-Bn}
    \sum_{j=0}^n
    \binom{n}{j}\e^{2Bj}
    g_\beta(dj,dn).
\end{equation}
It remains to estimate the asymptotics of the moment-generating function $g_\beta(dj,dn)$ as $n\to\infty$. To determine the exponential growth rate of
$g_\beta(k,m)$, Can's idea is to obtain recursions for $X(k,m)$ by exposing the partner of one vertex in $[k]$ in $\mathrm{CM}_m(1)$ at a time. 
We start with an arbitrary vertex $v \in [k]$ and pair it uniformly with another vertex. With probability $(k-1)/(m-1)$, it is paired with a vertex in $[k]$, which contributes no crossing edge between $[k]$ and $[m]\setminus [k]$, while the remaining pairings contribute $X(k-2,m-2)$ edges in distribution. With probability $(m-k)/(m-1)$, a vertex in $[m]\setminus [k]$ is chosen, contributing one crossing edge, while the remaining pairings contribute $X(k-1,m-2)$ edges in distribution. We thus obtain
\begin{equation}\label{eq:rec_1_g}
g_\beta(k,m)
=\frac{k-1}{m-1}g_\beta(k-2,m-2)
+\frac{\e^{-2\beta}(m-k)}{m-1}g_\beta(k-1,m-2).
\end{equation}
Alternatively, we may start with a vertex in $[m]\setminus[k]$.
It is paired with another vertex in $[m]\setminus[k]$ with probability $(m-k-1)/(m-1)$ and with a vertex in $[k]$ with probability $k/(m-1)$.
In the second case, the exposed pair contributes one crossing edge.
Consequently,
\begin{equation}\label{eq:rec_2}
g_\beta(k,m)
=\frac{m-k-1}{m-1}g_\beta(k,m-2)
+\frac{\e^{-2\beta}k}{m-1}g_\beta(k-1,m-2).
\end{equation}
Can next introduces the ratios
\begin{equation}\label{e:ratio}
h_\beta(k,m)
    :=
    \frac{g_\beta(k,m)}{g_\beta(k-1,m)}.
\end{equation}
The ratio $h_\beta(k,m)$ measures the multiplicative change in $g_\beta(k,m)$ when $k$ is increased by one. 
To derive a recursion for $h_\beta$, in \eqref{e:ratio}, we use \eqref{eq:rec_1_g} to replace $g_\beta(k,m)$ and \eqref{eq:rec_2} to replace $g_\beta(k-1,m)$ by expressions only involving $g_\beta(k-2,m-2)$ and $g_\beta(k-1,m-2)$.
Dividing through by $g_\beta(k-2,m-2)$ gives, for $2\leq k\leq m-1$,
\begin{equation}\label{eq:h_recursion}
h_\beta(k,m)
=\frac{(k-1)+\e^{-2\beta}(m-k)h_\beta(k-1,m-2)}
{\e^{-2\beta}(k-1)+(m-k)h_\beta(k-1,m-2)}.
\end{equation}
If $k/m\to t$, the ratios on both sides of \eqref{eq:h_recursion} converge to the same positive limit, which we denote by $h_\beta(t)$.
Formally passing to the limit gives the fixed-point equation
\begin{equation}\label{eq:h_fixed_point}
h_\beta(t)=\frac{t+\e^{-2\beta}(1-t)h_\beta(t)}{\e^{-2\beta}t+(1-t)h_\beta(t)}.
\end{equation}
Equivalently,
\begin{equation}\label{eq:h_quadratic}
(1-t)h_\beta(t)^2+\e^{-2\beta}(2t-1)h_\beta(t)-t=0.
\end{equation}
The unique positive solution of \eqref{eq:h_quadratic} is precisely $h_\beta(t)=f_\beta(t)$ from \eqref{e:f}.
The technical estimates in \cite[Lemma 2.1]{can} make this approximation rigorous and show, uniformly for $k\leq m/2$, that
\[
    h_\beta(k,m)
    =
    f_\beta\!\left(\frac{k-1}{m}\right)+O(m^{-1}).
\]
To go back from $h_\beta$ to $g_\beta$,
we can write $\log g_\beta$ as the telescoping sum 
\[
    \frac{1}{m}\log g_\beta(k,m)
    =
 \frac{1}{m}    \sum_{i=1}^k \log h_\beta(i,m),
\]
which is asymptotically a Riemann integral. Together with the symmetry
$g_\beta(k,m)=g_\beta(m-k,m)$, this yields, uniformly in $0\leq k\leq m$,
\[
    \frac{1}{m}\log g_\beta(k,m)
    =
    F_\beta\!\left(\frac{k}{m}\right)+o(1),
\]
with $F_{\beta}$ as in \eqref{eq:F}. 
In particular, uniformly in $j$,
\[
    \frac{1}{n}\log g_\beta(dj,dn)
    =
    dF_\beta(j/n)+o(1).
\]
Finally, passing to the limit $n \to \infty$, an application of the Laplace principle to the sum in \eqref{e:Z} that makes use of the previous estimate together with Stirling's formula for the binomial coefficient gives
\[
    \varphi^{\mathrm{an}}(\beta,B)
    =
    \frac{\beta d}{2}-B
    +
    \max_{0\leq t\leq1}
    \bigl\{
        H(t)+2Bt+dF_\beta(t)
    \bigr\}.
\]
\end{proof}
Can also uses this representation to prove
 that the annealed and quenched Ising pressures
agree \cite[Proposition~3.2]{can}. This gives an alternative analytic proof of the agreement that also applies to all
$d \ge 2$.

\paragraph{\bf Ising variational formula via analytic combinatorics.}
We next present the combinatorial representation of the annealed pressure for ferromagnetic Ising models on random $d$-regular graphs obtained by Borinsky, Ren and Wiesmann~\cite{BorinskyRenWiesmann} for $d \geq 3$. 
To formulate the result, we set 
\begin{equation}\label{e:kappa_tau_phi}
    \kappa:=\tanh(\beta),
    \qquad
    \tau:=\tanh(B),
    \qquad 
    \phi:=\arctan\sqrt{\tanh(\beta) }.
\end{equation}
For $\theta \in \mathbb{R}$, let 
\begin{equation}
    \mathcal{V}_d(\theta;\beta,B)
    :=
    \frac{(1+\kappa)^{d/2}}{2d!}
    \Bigl[
        (1+\tau)\cos^d(\theta-\phi)
        +
        (1-\tau)\cos^d(\theta+\phi)
    \Bigr].
\end{equation}
The main result of Borinsky, Ren and Wiesmann~\cite{BorinskyRenWiesmann} is then the following: 

\begin{theorem}
\label{t:combinatorial_Ising_rep}
Let $d\geq3$.
For every $\beta>0$ and $B\in\bbR$, the annealed Ising pressure on the random $d$-regular configuration model is
\begin{equation}\label{eq:combinatorial_Ising}
    \varphi^{\mathrm{an}-\mathrm{Ising}}(\beta,B)
    =
    \log 2+ f(\beta,B),
\end{equation}
where
\begin{equation}
\label{eq:combinatorial_Ising_pressure}
    f(\beta,B)
    =
    \log\bigl(\cosh(B)\bigr)
    +
    \frac{d}{2}\log\bigl(\cosh(\beta)\bigr)
    +
    \log\left(
        d!\max_{\theta\in\bbR}|\mathcal{V}_d(\theta;\beta,B)|
    \right).
\end{equation}
\end{theorem}

Whereas Can's formula is parametrized by the proportion of positive spins and is derived from a recursion for the moment generating function of the number of disagreeing edges, the approach of Borinsky, Ren and Wiesmann~\cite{BorinskyRenWiesmann} begins with the high-temperature expansion of the Ising partition function, originally due to van der Waerden~\cite{vanDerWaerden1941}. This expansion rewrites the partition function on a fixed graph as a weighted sum over edge subsets. After averaging over random regular graphs, these terms can be interpreted as weighted edge-bicolored regular graphs. The asymptotic enumeration of these graphs, obtained in \cite{BorinskyMeroniWiesmann}, reduces the thermodynamic limit to the one-dimensional optimization problem in \eqref{eq:combinatorial_Ising_pressure}.

\begin{proof}[Sketch of the proof of \Cref{t:combinatorial_Ising_rep}]
\textbf{Step 1: Reduction to a weighted combinatorial sum.}		We first express the annealed partition function as a weighted sum over isomorphism classes of $d$-regular multi-graphs. 

For a finite graph $G=(V,E)$ and a spin configuration $\bsigma\in\{-1,1\}^{V}$, recall that the Boltzmann weight is given by
\begin{equation}\label{e:combinatorial_Ising_weight}
    W_G({\bsigma})
    =
    \exp\left(
        \beta\sum_{\{i,j\}\in E}\sigma_i\sigma_j
        +
        B\sum_{i\in V}\sigma_i
    \right),
\end{equation}
so that $Z_G(\beta,B) =	\sum_{\bsigma\in\{-1,1\}^{V}}W_G({\bsigma}).$  The partition function is invariant under graph isomorphisms: If $G\simeq G'$, then the isomorphism induces a bijection between their spin configurations that preserves the corresponding Boltzmann weights.

We distinguish three types of objects. The configuration model is defined as a uniform pairing of labeled half-edges attached to the labeled vertex set $[n]$. Forgetting the half-edge labels produces a vertex-labeled $d$-regular multi-graph, while additionally forgetting the vertex labels produces an isomorphism class of unlabeled $d$-regular multi-graphs. 

Let $\calG_{n,d}$ denote the collection of isomorphism classes of unlabeled $d$-regular multi-graphs on $n$ vertices. For each $g\in\calG_{n,d}$, choose a representative $G\in g$. Since the partition function is invariant under graph isomorphisms, we may unambiguously write \( Z_g(\beta,B):=Z_G(\beta,B). \) We denote by $\aut(g)$ the full automorphism group of the corresponding half-edge configuration. Besides automorphisms induced by permutations of the vertices, this group includes permutations of parallel edges and exchanges of the two half-edges forming a loop. 

Define 
\begin{equation}\label{e:weighted_partition_sum}
Z_{n,d}(\beta,B) := \sum_{g\in\calG_{n,d}} \frac{Z_g(\beta,B)}{|\aut(g)|}.
\end{equation}
The configuration-model expectation of the partition function can be expressed as the following normalized weighted sum.
    
    \begin{lemma}\label{lem:ann_combi}
 For the Ising model on the random $d$-regular configuration model,
      \begin{equation}\label{e:annealed_weighted_sum}
		\bbE[Z_G(\beta,B)]
		=
		\frac{Z_{n,d}(\beta,B) }{\displaystyle\sum_{g\in\calG_{n,d}}|\aut(g)|^{-1}}.
	\end{equation}  
    \end{lemma}
	\begin{proof}[Proof of Lemma~\ref{lem:ann_combi}]
	 In \eqref{e:annealed_weighted_sum}, the expectation is taken over vertex-labeled graphs $G$ chosen from the configuration model after forgetting the labels of the half-edges. For a fixed vertex-labeled $d$-regular multi-graph $G$ on vertex set $[n]$, its probability under the configuration model is given by 
	\begin{equation}\label{CM_formula_with_label}
		\bbP(\mathrm{CM}_n=G)
		=
		\frac{1}{(dn-1)!!}
		\frac{(d!)^n}{\prod_{i\in[n]}2^{m_{ii}}\prod_{1\leq i\leq j\leq n}m_{ij}!}  ,
	\end{equation}
	where $m_{ij}$ is the number of edges between vertices $i$ and $j$, and $m_{ii}$ is the number of loops at vertex $i$; see \cite[Proposition 7.7]{vdH_2016}. The denominator \(\prod_{i\in[n]}2^{m_{ii}} \prod_{1\leq i\leq j\leq n}m_{ij}! \) accounts for the different half-edge pairings that give rise to the same vertex-labeled multi-graph: the factors $m_{ij}!$ account for permutations of parallel edges, while each loop contributes a factor $2$ from exchanging its two half-edges.

    Now let $g$ be the unlabeled isomorphism class of $G$. The number of distinct vertex labelings of a representative of $g$ is \( \frac{n!}{|\aut_v(G)|}, \) where $\aut_v(G)$ denotes the automorphism group induced on the vertex set. Since all these vertex labelings have the same probability under the configuration model, 
    \begin{equation}\label{CM_formula_without_label} 
    \bbP(\mathrm{CM}_n(d)\in g) = \frac{n!}{|\aut_v(G)|} \bbP(\mathrm{CM}_n(d)=G). 
    \end{equation} 
    Combining \eqref{CM_formula_with_label} and \eqref{CM_formula_without_label} gives 
    \begin{equation}\label{e:probability} 
    \bbP(\mathrm{CM}_n(d)\in g) = \frac{n!(d!)^n}{(dn-1)!!} \frac{1}{|\aut(g)|},
    \end{equation}
    where \( |\aut(g)| = |\aut_v(G)| \prod_{i\in[n]}2^{m_{ii}} \prod_{1\leq i\leq j\leq n}m_{ij}! \) is the order of the full automorphism group associated with the multi-graph, including vertex automorphisms, permutations of parallel edges, and exchanges of the two half-edges of loops. Thus, the probability of an isomorphism class $g$ is proportional to $|\aut(g)|^{-1}$. Since the partition function depends only on the isomorphism class, it follows that 
    \[ 
    \bbE[Z_G(\beta,B)] = \frac{ \displaystyle \sum_{g\in\calG_{n,d}} |\aut(g)|^{-1}Z_g(\beta,B) }{ \displaystyle \sum_{g\in\calG_{n,d}} |\aut(g)|^{-1} } = \frac{Z_{n,d}(\beta,B)} {\displaystyle\sum_{g\in\calG_{n,d}}|\aut(g)|^{-1}}. 
    \] 
	\end{proof} 
Since $Z_g(0,0)=2^n$ for every $g\in\calG_{n,d}$, we have \( Z_{n,d}(0,0) = 2^n\sum_{g\in\calG_{n,d}}|\aut(g)|^{-1}. \) Consequently, \Cref{lem:ann_combi} gives the exact identity 
\begin{equation}\label{eq:weighted_annealed_relation}
\bbE[Z_G(\beta,B)] = 2^n\frac{Z_{n,d}(\beta,B)}{Z_{n,d}(0,0)}.
\end{equation}
It follows that \( \varphi^{\mathrm{an}-\mathrm{Ising}}(\beta,B) = \log 2+f(\beta,B), \) where 
\begin{equation}\label{eq:def_f_combinatorial} 
f(\beta,B) := \lim_{n\to\infty} \frac{1}{n} \log\frac{Z_{n,d}(\beta,B)}{Z_{n,d}(0,0)}. 
\end{equation} 
It therefore remains to determine the exponential growth rate of the weighted sum $Z_{n,d}(\beta,B)$.

\textbf{Step 2: High-temperature expansion.} We first rewrite the Ising partition function using the high-temperature expansion. For this, recall $\kappa$ and $\tau$ from \eqref{e:kappa_tau_phi}, and moreover set
\begin{equation}
    \rho:=\cosh(\beta),
    \qquad \text{and} \qquad
    \eta:=\cosh(B).
\end{equation}
The identity
\begin{equation}
    \e^{t\sigma}
    =
    \cosh(t)\bigl(1+\sigma\tanh(t)\bigr),
    \qquad \sigma\in\{-1,1\},
\end{equation}
gives 
\begin{equation}
\begin{split}
    W_G({\bsigma})
    ={}&
    [\cosh(\beta)]^{|E|}
    [\cosh(B)]^{|V|}
    \prod_{\{i,j\}\in E}
    \bigl(1+\sigma_i\sigma_j\tanh(\beta)\bigr)\prod_{i\in V}
    \bigl(1+\sigma_i\tanh(B)\bigr).   
\end{split}
\end{equation}
Expanding the product over the edges amounts to choosing a subset $\gamma\subseteq E$. For a fixed $\gamma$, the spin $\sigma_v$ appears with power $\deg_\gamma(v)$. Here, $\deg_\gamma(v)$ denotes the degree of $v$ in the subgraph induced by $\gamma$, with each loop counted twice. Summing independently over the spins at the vertices gives 
\[ \sum_{\sigma_v\in\{-1,1\}} \sigma_v^{\deg_\gamma(v)} \bigl(1+\tau\sigma_v\bigr) = 2\tau^{\deg_\gamma(v)\bmod 2}. \]
Moreover, \( \kappa^{|\gamma|} = \prod_{v\in V} \kappa^{\deg_\gamma(v)/2}. \)
Since $G$ is $d$-regular, $|E|=dn/2$, and we obtain 
\begin{equation}\label{eq:cancellation_free_Ising} 
Z_g(\beta,B) = \sum_{\gamma\subseteq E_g} \prod_{v\in V_g}\lambda_v^{(g,\gamma)}, 
\end{equation}
where 
\begin{equation}\label{eq:effective_vertex_weight} 
\lambda_v^{(g,\gamma)} := 2\eta\rho^{d/2} \kappa^{\deg_\gamma(v)/2} \tau^{\deg_\gamma(v)\bmod 2}. 
\end{equation} 
This is the high-temperature expansion in the form needed below; see \cite[Proposition~2.3]{BorinskyRenWiesmann}.

Substituting \eqref{eq:cancellation_free_Ising} into the definition of $Z_{n,d}$ gives 
\begin{equation}\label{eq:double_sum_bicoloring} 
Z_{n,d}(\beta,B) = \sum_{g\in\calG_{n,d}} \frac{1}{|\aut(g)|} \sum_{\gamma\subseteq E_g} \prod_{v\in V_g}\lambda_v^{(g,\gamma)}. 
\end{equation}

\textbf{Step 3: Edge-bicolored graph interpretation.} We interpret a subset $\gamma\subseteq E_g$ as an edge bicoloring by assigning color $1$ to the edges in $\gamma$ and color $0$ to those in $E_g\setminus\gamma$. If \(\deg_1(v):=\deg_\gamma(v), \deg_0(v):=d-\deg_\gamma(v), \) then the weight of $v$ depends only on the pair $(\deg_0(v),\deg_1(v))$. More precisely, setting 
\begin{equation}\label{eq:Lambda_Ising}
\Lambda_{a,b} := 2\eta\rho^{d/2}\kappa^{b/2}\tau^{b\bmod 2}, \qquad a+b=d, 
\end{equation} 
we have \(\lambda_v^{(g,\gamma)} = \Lambda_{\deg_0(v),\deg_1(v)}. \) Thus, \eqref{eq:double_sum_bicoloring} is a weighted enumeration of $d$-regular edge-bicolored multi-graphs.

The expression in \eqref{eq:double_sum_bicoloring} still sums separately over an uncolored isomorphism class $g$ and over all subsets of the edge set of a chosen representative. We next reorganize it as a single sum over isomorphism classes of edge-bicolored multi-graphs. 

For a fixed $g\in\calG_{n,d}$, the group $\aut(g)$ acts on $2^{E_g}$ by sending an edge subset $\gamma$ to its image under an automorphism. Two subsets lie in the same orbit precisely when they define isomorphic edge-bicolored multi-graphs. Moreover, the weight \[ \prod_{v\in V_g}\lambda_v^{(g,\gamma)} \] is constant on every orbit. 

Let $\calO(g)$ denote the collection of orbits of this action. For an orbit $\calO_\gamma\in\calO(g)$, let $\aut(g,\gamma)$ denote the subgroup of $\aut(g)$ that preserves $\gamma$, or equivalently, preserves the two edge colors. By the orbit--stabilizer theorem, \[ |\calO_\gamma| = \frac{|\aut(g)|}{|\aut(g,\gamma)|}. \] Therefore, \begin{align} Z_{n,d}(\beta,B) &= \sum_{g\in\calG_{n,d}} \sum_{\calO_\gamma\in\calO(g)} \frac{|\calO_\gamma|}{|\aut(g)|} \prod_{v\in V_g}\lambda_v^{(g,\gamma)} \nonumber\\ &= \sum_{g\in\calG_{n,d}} \sum_{\calO_\gamma\in\calO(g)} \frac{1}{|\aut(g,\gamma)|} \prod_{v\in V_g} \Lambda_{\deg_0(v),\deg_1(v)}. \label{eq:bicolored_isomorphism_sum} \end{align} Hence, $Z_{n,d}(\beta,B)$ is a weighted sum over isomorphism classes of $d$-regular edge-bicolored multi-graphs, with each class weighted by the reciprocal of the size of its color-preserving automorphism group.

\textbf{Step 4: Asymptotic enumeration.} We can now apply the asymptotic enumeration results of Borinsky, Meroni and Wiesmann~\cite{BorinskyMeroniWiesmann}. Their result applies to weighted sums of precisely the form \eqref{eq:bicolored_isomorphism_sum}, where the weight of a vertex depends only on the numbers of incident edges of the two colors. 

To state it, let \( r:=|E|-|V|. \) For a $d$-regular graph on $n$ vertices, 
\begin{equation}\label{eq:r_n_relation} 
r = \frac{dn}{2}-n = \frac{d-2}{2}n. 
\end{equation}
Equivalently, the usual Euler characteristic $|V|-|E|$ is $-r$. Thus, taking $n\to\infty$ with fixed $d\geq3$ is equivalent to taking $r\to\infty$.

	\begin{proposition}[{\cite[Proposition 5.1]{BorinskyMeroniWiesmann}}]
	\label{p:BMW_asymptotic_enumeration}
	Let $\calG^d_{-r}$ be the set of isomorphism classes of $d$-regular edge-bicolored graphs of Euler characteristic $-r$, and let
	\begin{equation}
		A_r
		:=
		\sum_{G\in\calG^d_{-r}}
		\frac{1}{|\aut(G)|}
		\prod_{v\in V}\Lambda_{\deg_0(v),\deg_1(v)},
	\end{equation}
	where $\deg_0(v)$ and $\deg_1(v)$ are the numbers of edges of the two colors incident to $v$, and $\Lambda_{a,b}\in\bbR$ is fixed for every $a,b\geq0$ with $a+b=d$.
	Define
	\begin{equation}\label{eq:function_V}
		\mathcal{V}(x,y)
		:=
		\sum_{a+b=d}\Lambda_{a,b}\frac{x^ay^b}{a!b!},
	\end{equation}
	let $\Phi$ be the set of global maximizers of $|\mathcal{V}|$ on the unit circle $S^1$, and set
	\begin{equation}\label{e:function_B}
		\mathcal B(x,y)
		:=
		d^2-
		\frac{\partial_x^2\mathcal{V}(x,y)+\partial_y^2\mathcal{V}(x,y)}{\mathcal{V}(x,y)}.
	\end{equation}
	Let $M=d/(d-2)$ and $K=2/(d-2)$.
	If $\mathcal B(x,y)\neq0$ for every $(x,y)\in\Phi$, then
	\begin{equation}\label{eq:BMW_proposition_5_1}
		A_r
		\sim
		\begin{cases}
		\displaystyle
		\frac{1}{2\sqrt{2}\pi}
		d^{rM+1/2}K^{r-1/2}\Gamma(r)
		\sum_{(x,y)\in\Phi}
		\frac{\mathcal{V}(x,y)^{rK}}{\sqrt{\mathcal B(x,y)}},
		& rK,rM\in\bbZ,\\[2ex]
		0,
		& \text{otherwise}.
		\end{cases}
	\end{equation}
	\end{proposition}
    
    \textbf{Step 5: Specialization to the Ising weights.} 
    We now apply \Cref{p:BMW_asymptotic_enumeration} to the weighted sum in \eqref{eq:bicolored_isomorphism_sum}. By \eqref{eq:r_n_relation},
    \[ rK=n, \qquad rM=\frac{dn}{2}, \] 
    where $K=2/(d-2)$ and $M=d/(d-2)$. Thus, the integrality conditions in \Cref{p:BMW_asymptotic_enumeration} are satisfied precisely for the admissible values of $n$, namely those for which $dn$ is even. 
    
    The effective vertex weights in \eqref{eq:Lambda_Ising} are \[ \Lambda_{d-j,j} = 2\eta\rho^{d/2} \kappa^{j/2}\tau^{j\bmod 2}, \qquad 0\leq j\leq d. \] Define the homogeneous polynomial \begin{equation}\label{eq:Ising_polynomial_xy} \mathcal V_d(x,y;\beta,B) := \sum_{j=0}^d \kappa^{j/2}\tau^{j\bmod 2} \frac{x^{d-j}y^j}{(d-j)!j!}. \end{equation} The polynomial $\mathcal V$ in \Cref{p:BMW_asymptotic_enumeration} is therefore 
    \begin{equation}\label{eq:BMW_polynomial_Ising}
    \mathcal V(x,y) = 2\eta\rho^{d/2}\mathcal V_d(x,y;\beta,B). 
    \end{equation} 

On the unit circle, write $x=\cos\theta$ and $y=\sin\theta$. Separating the even and odd values of $j$ in \eqref{eq:Ising_polynomial_xy} and applying the binomial theorem yields 
    \begin{equation}\label{eq:Ising_polynomial_theta} 
    \begin{split} \mathcal V_d(\theta;\beta,B) &:= \mathcal V_d(\cos\theta,\sin\theta;\beta,B)\\ &= \frac{(1+\kappa)^{d/2}}{2d!} \left[ (1+\tau)\cos^d(\theta-\phi) + (1-\tau)\cos^d(\theta+\phi) \right], \end{split} 
    \end{equation} 
    where $\phi=\arctan\sqrt{\kappa}$. This is precisely the function introduced  before \Cref{t:combinatorial_Ising_rep}.
    
    \textbf{Step 6: Extraction of the annealed pressure.}
    Let $\Phi_d(\beta,B)$ denote the set of global maximizers of $|\mathcal V_d(x,y;\beta,B)|$ on the unit circle. Multiplication by the positive constant $2\eta\rho^{d/2}$ in \eqref{eq:BMW_polynomial_Ising} does not change the set of maximizers. It also leaves the function $\mathcal B$ in \eqref{e:function_B} unchanged, since this constant cancels in the quotient defining $\mathcal B$. 
    
    By \cite[Lemmas~3.4 and~3.5]{BorinskyRenWiesmann}, the global maxima are nondegenerate away from the critical point. Hence, for noncritical parameters and admissible values of $n$, \Cref{p:BMW_asymptotic_enumeration} gives 
    \begin{equation}\label{eq:Zn_d_combinatorial_asymptotic} 
    \begin{split} 
    Z_{n,d}(\beta,B) \sim{}& C_{n,d} \bigl(2\eta\rho^{d/2}\bigr)^n \sum_{(x,y)\in\Phi_d(\beta,B)} \frac{\mathcal V_d(x,y;\beta,B)^n} {\sqrt{\mathcal B_d(x,y;\beta,B)}}, 
    \end{split} 
    \end{equation}
    where \[ C_{n,d} := \frac{1}{2\sqrt{2}\pi} d^{dn/2+1/2} \left(\frac{2}{d-2}\right)^{r-1/2} \Gamma(r). \] 
    Here, $\mathcal B_d$ denotes the function $\mathcal B$ associated with $\mathcal V_d$. If $dn$ is odd, no $d$-regular graph on $n$ vertices exists, and hence $Z_{n,d}(\beta,B)=0$.
    
    At $\beta=B=0$, we have $\kappa=\tau=0, \eta=\rho=1,$  and therefore $\mathcal V_d(x,y;0,0) = \frac{x^d}{d!}.$  In particular, 
    \[ \max_{(x,y)\in S^1} |\mathcal V_d(x,y;0,0)| = \frac{1}{d!}. \] 
   Applying the same asymptotic formula to $Z_{n,d}(0,0)$ and taking the ratio, the common factor $C_{n,d}$ cancels. For admissible values of $n$, \Cref{p:BMW_asymptotic_enumeration} shows that the exponential growth rate is determined by the maximum of $|\mathcal V_d|$ on the unit circle. Consequently,
    \begin{equation}\label{eq:combinatorial_ratio_limit} 
     \lim_{\substack{n\to\infty\\dn\text{ even}}} \frac{1}{n} \log\frac{Z_{n,d}(\beta,B)}{Z_{n,d}(0,0)} = \log\eta + \frac{d}{2}\log\rho + \log\left( d!\max_{\theta\in\bbR} |\mathcal V_d(\theta;\beta,B)| \right). 
    \end{equation} 
    Away from the critical point, this identity follows from the preceding asymptotic formulas. At the critical point, the same formula for the exponential growth rate follows by continuity in $(\beta,B)$.
    
    By \eqref{eq:def_f_combinatorial}, the left-hand side of \eqref{eq:combinatorial_ratio_limit} is $f(\beta,B)$. Substituting $\eta=\cosh(B)$ and $\rho=\cosh(\beta)$ gives \eqref{eq:combinatorial_Ising_pressure}. Finally, \eqref{eq:weighted_annealed_relation} yields \eqref{eq:combinatorial_Ising}, completing the proof. 
    \end{proof}

\end{document}